\documentclass[11pt,reqno]{amsart}

\usepackage[sort]{natbib}

\usepackage[T1]{fontenc}
\usepackage{enumitem}
\usepackage{hyperref}
\hypersetup{
	colorlinks   = true,
	citecolor    = blue,
	linkcolor    = blue
}

\usepackage{amsmath,amsfonts,amsthm,color,tikz, comment, xcolor, xfrac,amssymb}
\usepackage{algpseudocode} 
\usepackage{algorithm}
\usepackage{booktabs}



\usepackage[normalem]{ulem}
\usepackage{bbm}
\usepackage{mathtools}
\usepackage{shuffle}
\usepackage{pdfsync}
\usepackage[font={scriptsize}]{subcaption}
\renewcommand{\thesubfigure}{\Alph{subfigure}}
\usepackage{environ}

\usepackage[left=1in, right=1in, top=1.1in,bottom=1.1in]{geometry}
\newcommand{\der}{\delta}

\newcommand{\E}{\mathbb E}
\newcommand{\R}{\mathbb R}

\newcommand{\1}{\mathbbm 1}

\newcommand{\bp}{\mathbf{P}}

\newcommand{\cd}{\mathcal D}

\newcommand{\cf}{\mathcal F}

\newcommand{\ch}{\mathcal H}

\newcommand{\cm}{\mathcal M}

\newcommand{\al}{\alpha}

\newcommand{\ep}{\varepsilon}
\newcommand{\io}{\iota}
\newcommand{\ga}{\gamma}

\newcommand{\la}{\lambda}

\newcommand{\si}{\sigma}
\newcommand{\te}{\vartheta}

\newcommand{\lp}{\left(}
\newcommand{\rp}{\right)}

\newcommand{\lcl}{\left\{}
\newcommand{\rcl}{\right\}}

\newtheorem{theorem}{Theorem}[section]

\newtheorem{corollary}[theorem]{Corollary}

\newtheorem{definition}[theorem]{Definition}

\newtheorem{assumption}[theorem]{Assumption}
\newtheorem{lemma}[theorem]{Lemma}

\newtheorem{proposition}[theorem]{Proposition}

\theoremstyle{remark}
\newtheorem{remark}[theorem]{Remark}

\theoremstyle{remark}

\theoremstyle{definition}
\newtheorem{example}[theorem]{Example}

\newcommand{\bean}{\begin{eqnarray*}}
	\newcommand{\eean}{\end{eqnarray*}}
\newcommand{\ben}{\begin{enumerate}}
	\newcommand{\een}{\end{enumerate}}
\newcommand{\beq}{\begin{equation}}
	\newcommand{\eeq}{\end{equation}}

\title[Volterra-Skorokhod reflection and loss queues]{Volterra-Skorokhod Reflection and Fluid Limits of Loss Queues}

\author[MINGRUI WANG]{MINGRUI WANG}
\address{(M. Wang) Harold and Inge Marcus Department of Industrial and Manufacturing
Engineering, The Pennsylvania State University, University Park, PA 16802 United States\\
}
\email{mvw5822@psu.edu}

\author[PRAKASH CHAKRABORTY]{PRAKASH CHAKRABORTY}
\address{(P. Chakraborty) Harold and Inge Marcus Department of Industrial and Manufacturing
Engineering, The Pennsylvania State University, University Park, PA 16802 United States\\
}
\email{prakashc@psu.edu}
\thanks{P. Chakraborty is partially supported by the National Science Foundation under grant DMS-2153915 and Korb Early Career Professorship.}

\date{\today}

\begin{document}

\begin{abstract} 
In the $G_t/G/n/n$ loss queue, a customer blocked at time $u$ remains missing from the system for as long as it would have been served. A fraction $K(t-u)$ of the blocking at $u$ therefore still affects the occupancy at time $t$, where $K=1-G$ is the service survival function. We show that, for every arrival process satisfying a functional strong law of large numbers, the fluid limit of this queue is a reflection with kernel $K$ of the fluid limit of the associated infinite-server queue.
The reflection is defined by a one-sided reflection problem on $(-\infty,1]$ in which the regulator is convolved with a memory kernel. Given a c\`adl\`ag input $\rho$ and a kernel $K$ with $K(0)=1$, we seek a pair $(\hat\rho,\eta)$ such that $\hat\rho=\rho-\int_0^\cdot K(\cdot-u)\,d\eta(u)$ takes values in $(-\infty,1]$, $\eta$ is nondecreasing, and $\eta$ increases only when $\hat\rho=1$. We refer to this problem as the Volterra-Skorokhod problem. We show existence and uniqueness of the solution and that the solution map is Lipschitz continuous in the uniform norm and nonanticipating. It therefore applies to stochastic inputs path by path. Unlike the classical Skorokhod map, this map is not monotone. For nonincreasing kernels, however, the regulator is the smallest admissible one and monotone in the input. In the queue, this orders the cumulative fluid losses under ordered arrival rates. Furthermore, since arrivals enter the queue through the reflection kernel, the regulator is a thinning of the arrivals.
\end{abstract}

\maketitle

\section{Introduction}
The Skorokhod problem is the standard device for constraining a stochastic process to a half-line. In the form used here, given a c\`adl\`ag path $\psi$ with $\psi(0)\leq1$, one seeks $(\phi,\eta)$ such that $\phi=\psi-\eta$ takes values in $(-\infty,1]$, $\eta$ is nondecreasing with $\eta(0)=0$, and $\eta$ increases only at times when $\phi=1$. The solution is explicit, $\eta(t)=\sup_{s\leq t}(\psi(s)-1)^+$, and the induced reflection map $\Gamma\colon\psi\mapsto\phi$ is Lipschitz in the uniform norm \cite{skorokhod1961,tanaka1979,Kruk2007}.

Since Skorokhod's construction of diffusions with reflecting boundaries \cite{skorokhod1961}, this map has organized much of the theory of constrained processes: reflected Brownian motion \cite{harrison_reiman1981} and its role in heavy-traffic analysis, multidimensional and obliquely reflected diffusions analyzed through Lipschitz estimates and convex duality \cite{dupuis_ishii1991,dupuis_ramanan1999}, the doubly reflected map on $[0,a]$ with its explicit formula \cite{Kruk2007}, the extended Skorokhod map \cite{ramanan2006}, and the generalized drift Skorokhod problem \cite{reed_ward_zhan2013}. Continuity of $\Gamma$ in the Skorokhod $J_1$ and $M_1$ topologies underlies the continuous-mapping route to functional limit theorems \cite{whitt2001reflection,whitt2002}.

In each of these settings the regulator acts instantaneously. Regulation exerted at time $u$ lowers the state by the same amount at every later time. This is appropriate when the dynamics are memoryless, but capacity-constrained queues with general service are not. A customer admitted at time $u$ holds a server for a random duration, and hence the effect of the boundary regulation exerted at $u$ persists according to the service-time distribution. With $K=1-G$ the service-time survival function, a fraction $K(t-u)$ of the regulation at $u$ is still felt at time $t$, and the constraint is enforced by the convolution of $\eta$ with $K$.

This leads to the problem we study. Given an input $\rho\in\cd([0,\infty),\R)$ with $\rho(0)\leq1$ and a kernel $K$ with $K(0)=1$, the \emph{Volterra-Skorokhod problem with kernel $K$} asks for $(\hat\rho,\eta)$ with
\[
\hat\rho(t)=\rho(t)-\int_0^t K(t-u)\,d\eta(u)\leq1,\qquad t\geq0,
\]
$\eta$ nondecreasing with $\eta(0)=0$, and $\int_0^\infty\1_{\{\hat\rho(t)<1\}}\,d\eta(t)=0$. For $K\equiv1$ this is the classical Skorokhod reflection problem. For a general kernel the convolution couples the regulation across time, and the explicit sup-formula for the classical case no longer holds.

Reflection problems with Volterra dynamics were recently considered in \cite{ascione2025} for equations $z(t)=y(t)+\int_0^t k(t-s)b(s,z(s))\,ds+a(t)$, in which the kernel acts on the drift $b$ and the regulator $a$ enters additively. The constraint then remains instantaneous, the classical Skorokhod map applies pathwise, and existence follows by Picard iteration from its Lipschitz property. In our problem the regulator itself is convolved with the kernel. The additive term $a(t)$ is replaced by $-\int_0^t K(t-u)\,d\eta(u)$, and regulation at time $u$ keeps acting at all later times. The loss queue produces this memory, and it is the source of the analytical difficulty.

Our strategy is to reduce the Volterra-Skorokhod problem to the ordinary one. Let $\mu_K$ be the measure given by $\mu_K((s,t])=K(t)-K(s)$, and let $r$ be the resolvent of $-\mu_K$. We show that any solution $\hat\rho$ must be a fixed point of $\Gamma\circ\psi_\rho^K$, where $\psi_\rho^K$ is an affine map built from $r$, and conversely that this fixed point solves the problem. Combining the Lipschitz property of $\Gamma$ with the Bielecki weighted-norm renorming \cite{bielecki1956,Bessenyei2021} makes $\Gamma\circ\psi_\rho^K$ a contraction, from which existence, uniqueness, and Lipschitz continuity of the Volterra-Skorokhod map $\Gamma_K$ follow. The resolvent calculus we use is standard for Volterra equations with measure-valued kernels \cite{Gripenberg1990,Brunner2017}.

We summarize the contributions. Our main application concerns the $G_t/G/n/n$ loss queue. For every arrival process satisfying a functional strong law of large numbers (Assumption~\ref{asm:queue0}), its fluid limit is $\Gamma_K(\rho)$ with $K=1-G$, where $\rho$ is the fluid limit of the corresponding infinite-server queue (Theorem~\ref{thm:fluid}). The proof controls the error from replacing the data-dependent regulation kernel by $K$ through a uniform law of large numbers over a Vapnik-Chervonenkis class. The regulator of the limit is the fluid limit of the scaled number of blocked customers, and a larger arrival rate yields a larger cumulative loss (Corollary~\ref{cor:loss-order}). Because arrivals enter the queue through the reflection kernel, the regulator is a thinning of the arrivals with an admission fraction characterized by a free-boundary equation (Theorem~\ref{prop:thinning} and Corollary~\ref{thm:rho-sol}). 

These results rest on properties of $\Gamma_K$ that we establish for general kernels. We prove existence and uniqueness for the Volterra-Skorokhod problem, with a resolvent representation of the solution (Theorem~\ref{thm:V-S-sol}), and Lipschitz continuity of $\Gamma_K$ in the uniform norm (Proposition~\ref{prop:lip}). The convolution changes the order structure of the problem. Unlike the classical map, $\Gamma_K$ is not monotone in the input (Example~\ref{prop:non-mono}). For nonincreasing $K$ the regulator is nevertheless the smallest admissible one (Theorem~\ref{prop:eta-min}). It is therefore monotone in the input (Corollary~\ref{thm:eta-mono}), which gives a stochastic ordering of the loss (Corollary~\ref{cor:eta-st}).

Although $\Gamma_K$ is deterministic, it applies to random inputs path by path. Since it is nonanticipating, it maps adapted c\`adl\`ag inputs to adapted reflected processes, with moment bounds inherited from the input and continuous paths when the input and $K$ are continuous (Proposition~\ref{prop:reflected-process}).

The closest work is \cite{wang2025fluid}, which proves functional strong laws of large numbers for the $M_t/G/n/n$ queue and the finite-buffer $M_t/G/n/(n+b_n)$ queue, and characterizes the limits through Volterra integral equations with discontinuous coefficients. The present paper differs from it in three respects. We study the Volterra-Skorokhod map as a deterministic object for general kernels, apart from any queue, and prove its well-posedness, Lipschitz continuity, nonanticipation and order properties. Our fluid limit (Theorem~\ref{thm:fluid}) requires only a functional strong law of large numbers for the arrivals, instead of nonhomogeneous Poisson input, and it comes with the pathwise approximation $\|\bar N^n-\Gamma_K(\rho^n)\|_T\to0$. Finally, 
Corollary~\ref{thm:rho-sol} identifies the Volterra integral equation of \cite{wang2025fluid} with the Volterra-Skorokhod problem. Through this identification, the fraction $w$ of fluid arrivals admitted
exists by Theorem~\ref{thm:V-S-sol} instead of by the queueing limit, and Theorem~\ref{prop:thinning} extends the thinning representation to arrival measures with atoms. We do not treat finite buffers.

The paper is organized as follows. Section~\ref{sec:prelim} fixes notation and recalls the convolution of signed measures and the Bielecki norm. Section~\ref{sec:vsp} formulates the Volterra-Skorokhod problem, and proves unique solvability,  Lipschitz continuity and the order properties of $\Gamma_K$. Section~\ref{sec:queue} applies the map to the $G_t/G/n/n$ loss queue, and represents its fluid limit through a thinning of the arrivals. Section~\ref{sec:numeric} reports numerical experiments, and Section~\ref{sec:conclusion} concludes.

\section{Preliminaries}\label{sec:prelim}

\subsection{Function and measure spaces}
We denote by $\mathcal D([0,T],\mathbb R)$ the space of c\`adl\`ag (right continuous with left limits) functions from $[0,T]$ to $\mathbb R$. When no confusion is possible, we write simply $\mathcal D([0,T])$. We use $|\cdot|$ for the absolute value of real numbers. For $f\in\mathcal D([0,T])$, we denote the uniform norm
$$
\|f\|_T:=\sup_{t\in[0,T]} |f(t)|.
$$
For a Borel set $J\subseteq \R$, we denote by $M(J)$ the space of finite signed Borel measures on $J$. If $\mu\in M(J)$, then $|\mu|$ denotes its total variation measure. We denote by $M_{\mathrm{loc}}(J)$ the space of signed Borel measures $\mu$ on $J$ such that
$$
|\mu|(C)<\infty,\quad \text{ for every compact set } C\subseteq J.
$$
For $F\in\mathcal D([0,T])$ of bounded variation, we denote by $\mu_F$ the signed measure associated to $F$, which is determined by
\beq\label{eq:mu_F}
\mu_F(\{0\})=0,\quad\mu_F((s,t])=F(t)-F(s),\quad 0\leq s<t\leq T.
\eeq
When $F$ is of bounded variation, we will abuse the notation $dF:=d\mu_F$.

Unless stated otherwise, an integral $\int_a^b$ against a measure is taken over the closed interval $[a,b]$. By \eqref{eq:mu_F}, $\mu_F(\{0\})=0$ for every $F$ of bounded variation, and hence $\int_0^tg\,d\mu_F=\int_{(0,t]}g\,d\mu_F$. For such integrators our convention coincides with the convention $\int_{0+}^t$ customary for c\`adl\`ag integrators, and a nonzero value $F(0)$ never creates an atom at the origin.

\subsection{Convolution}\label{sec:conv}
For $f,g\in L^1_{\mathrm{loc}}(\R_+)$, we define their convolution by
$$
(f*g)(t):=\int_0^t f(t-s)g(s)ds.
$$
For $\mu\in M_{\mathrm{loc}}(\R_+)$ and for a locally bounded Borel function $f:\mathbb R_+\to\mathbb R$, we define their convolution by
$$
(\mu*f)(t)=(f*\mu)(t):=\int_{[0,t]} f(t-s)d\mu(s),\qquad t\geq 0.
$$
\begin{remark}
For more details about convolution of signed measures we refer the reader to \cite[Section 8.6]{folland1999real} and \cite[Chapters 3 and 4]{Gripenberg1990}.
\end{remark}
The following result follows from dominated convergence and will be used in the sequel.
\begin{lemma}\label{lem:conv-cadlag}
For $\mu\in M_{\mathrm{loc}}(\R_+)$ and $f\in\cd([0,T])$, we have $\mu*f\in\cd([0,T])$.
\end{lemma}
For $\mu,\nu\in M_{\mathrm{loc}}(\R_+)$, their convolution $\mu*\nu \in M_{\mathrm{loc}}(\R_+)$ is defined on each finite interval as follows: for every $S<\infty$ and every
$E\in\mathcal B([0,S])$,
$$
(\mu*\nu)(E):=\int_{[0,S]}\int_{[0,S]}\mathbf \1_E(s+u)d\mu(s)d\nu(u).
$$

\subsection{Weighted norm}\label{sec:bielecki}
We briefly mention the Bielecki renorming technique \cite{bielecki1956}. Define the weighted norm for $\lambda>0$
\begin{equation*}
	\|f\|_{\lambda,T}:=\sup _{t \in[0, T]}\left(e^{-\lambda t}|f(t)|\right).
\end{equation*}
For any function $f \in \mathcal{D}([0, T])$, we have:
\begin{equation}\label{eq:w-norm-eq}
	e^{-\lambda T}\|f\|_T \leq\|f\|_{\lambda,T} \leq\|f\|_T.
\end{equation}
Hence $\|\cdot\|_{\lambda,T}$ and $\|\cdot\|_T$ are equivalent norms, and $(\cd([0,T]),\|\cdot\|_{\lambda,T})$ is complete \cite[Proposition 3]{Bessenyei2021}.

\begin{lemma}\label{lem:w-norm-bd}
Let $\mu\in M_{\mathrm{loc}}([0,T])$ and $f\in\cd([0,T])$. Then
\begin{equation}\label{eq:w-norm-bd}
\|\mu*f\|_{\lambda,T}\leq \|f\|_{\lambda,T} R_\mu(\lambda),
\end{equation}
where $R_\mu(\lambda)=\int_0^T e^{-\lambda s} d|\mu| (s)$. If, moreover $|\mu|(\{0\})=0$, then $\lim _{\lambda \rightarrow \infty} R_\mu(\lambda)=0$.
	\end{lemma}
\begin{proof}
For any $f\in\cd([0,T])$ and $t\in[0,T]$, we have
\begin{equation*}
			e^{-\lambda t}|(\mu * f)(t)| \leq e^{-\lambda t} \int_0^t|f(t-s)| d|\mu| (s) =\int_0^t\left(e^{-\lambda(t-s)}|f(t-s)|\right) e^{-\lambda s}d|\mu| (s).
\end{equation*}
Since $e^{-\lambda(t-s)}|f(t-s)| \leq\|f\|_{\lambda,T}$, we obtain:
\begin{equation*}
			e^{-\lambda t}|(\mu * f)(t)| \leq\|f\|_{\lambda,T} \int_0^t e^{-\lambda s} d|\mu| (s) \leq\|f\|_{\lambda,T} R_\mu(\lambda).
\end{equation*}
Taking the supremum over $t\in[0,T]$ we get \eqref{eq:w-norm-bd}. Recall that $R_\mu(\lambda)=\int_0^T e^{-\lambda s} d|\mu| (s)=|\mu|(\{0\})+\int_{(0,T]}e^{-\lambda s} d|\mu| (s)$. Since $|\mu|$ is finite on compact intervals and $ e^{-\lambda s}\rightarrow 0$ as $\lambda\rightarrow\infty$, by dominated convergence we have $\lim _{\lambda \rightarrow \infty} R_\mu(\lambda)=0$ if $|\mu|(\{0\})=0$.
\end{proof}
\section{Volterra-Skorokhod problem}\label{sec:vsp}
Without loss of generality we present the reflection problem on $(-\infty, 1]$. 
The problem on $(-\infty,a]$ reduces to this case on replacing the input $\rho$ by $\rho+1-a$. We recall the classical Skorokhod problem.
\begin{definition}[Skorokhod Problem on \({(-\infty,1]}\)]\label{def:SP}
Given a function $\psi \in \mathcal{D}([0, \infty), \mathbb{R})$ with $\psi(0)\leq 1$, we say that a pair of functions $(\phi, \eta) \in \mathcal{D}^2([0, \infty), \mathbb{R})$ satisfies the Skorokhod problem for $\psi$ if the following four conditions are satisfied:
\begin{enumerate}[label=\arabic*.]
	\item $\phi(t)=\psi(t)-\eta(t)$ for $t \geq 0$,
	\item $\phi(t) \leq 1$ for $t \geq 0$,
	\item $\eta$ is nondecreasing with $\eta(0)=0$,
	\item $\int_0^{\infty} \1_{\{\phi(t)<1\}} d \eta(t)=0$.
\end{enumerate}
\end{definition}

For each $\psi \in \mathcal{D}([0, \infty), \mathbb{R})$ with $\psi(0)\leq 1$, the Skorokhod problem has a unique solution $(\phi, \eta)$ given by $\eta(t) = \sup_{s \leq t} (\psi(s)-1)^{+}$ \cite{skorokhod1961, Kruk2007}. We write $\Gamma(\psi)=\phi$ and call $\Gamma$ the Skorokhod map.

\begin{definition}[Volterra-Skorokhod problem with kernel $K$ on \({(-\infty,1]}\)]\label{def:VSP}
Given a function $\rho \in \mathcal{D}([0, \infty), \mathbb{R})$ with $\rho(0)\leq 1$ and a c\`adl\`ag function of bounded variation $K: [0,\infty)\rightarrow \R_+$ with $K(0)\neq 0$, we say that a pair of functions $(\hat{\rho}, \eta) \in \mathcal{D}^2([0, \infty), \mathbb{R})$ satisfies the Volterra-Skorokhod problem with kernel $K$ for $\rho$ if the following four conditions are satisfied:
\begin{enumerate}[label=\arabic*.]
	\item The reflected function $\hat{\rho}$ is obtained from $\rho$ by subtracting a Volterra integral of $\eta$:
    \beq\label{eq:VS-reg}
    \hat{\rho}(t)=\rho(t)-\int_0^t K(t-u) d \eta(u),
    \text{ for }t \geq 0,
    \eeq
	\item $\hat{\rho}(t) \leq 1$ for $t \geq 0$,
	\item $\eta$ is nondecreasing with $\eta(0)=0$,
	\item $\int_0^{\infty} \1_{\{\hat{\rho}(t)<1\}} d \eta(t)=0$.
\end{enumerate}
\end{definition}
\begin{remark}
The assumption that the kernel $K$ is nonnegative is not required for the following  unique solvability result of the problem. However, we retain this assumption because it matches the interpretation as a reflection. When $K$ is negative, past regulation would raise the state instead of lowering it.
\end{remark}

\begin{remark}\label{rk:time-zero}
By \eqref{eq:mu_F}, $d\eta$ has no atom at $0$, and Condition~1 gives $\hat\rho(0)=\rho(0)$. The regulator therefore cannot act at time $0$, and $\rho(0)\leq1$ is necessary for Condition~2. For general $\rho$ we allow an initial push, that is, $\eta(0)\geq0$ with $d\eta$ having the atom $\eta(0)$ at $0$. Conditions~2 and~4 at $t=0$ then force $\eta(0)=c:=(\rho(0)-1)^+/K(0)$. As $\int_0^t K(t-u)\,d\eta(u)=cK(t)+\int_0^tK(t-u)\,d(\eta-c)(u)$, the pair $(\hat\rho,\eta)$ solves the problem for $\rho$ if and only if $(\hat\rho,\eta-c)$ solves it for $\rho-cK$, and $(\rho-cK)(0)\leq1$. The results below therefore extend to general $\rho$
\end{remark}

Without loss of generality we assume $K(0) = 1$ and by Remark~\ref{rk:time-zero}, $\rho(0) \leq 1$ throughout, and these conditions are not repeated in the statements below. If $K(0)\neq 1$, we can replace $K$ by $\tilde{K}(t)=K(t)/K(0)$, in which case $(\hat{\rho},K(0)\eta)$ solves the Volterra-Skorokhod problem with kernel $\tilde{K}$ for $\rho$. 

\begin{theorem}\label{thm:V-S-sol}
Given any $\rho \in \mathcal{D}([0, T], \mathbb{R})$ and a 
kernel $K$ as in Definition~\ref{def:VSP} with $K(0)=1$, there exists a unique solution $(\hat{\rho}, \eta) \in \mathcal{D}^2([0, T], \mathbb{R})$ to the Volterra-Skorokhod problem in Definition~\ref{def:VSP}. Moreover, $\hat{\rho}$ is the unique fixed point of $\Gamma(\psi_{\rho}^{K}(\cdot))$ where
\begin{equation}\label{eq:psi_Thm}
	\psi_{\rho}^{K}(f)=r*\rho+\rho-r*f,\quad \text{for } f \in\mathcal{D}([0, T], \mathbb{R})
\end{equation}
Here, $r\in M_{\mathrm{loc}}([0,T])$ is the resolvent kernel of $-\mu_K$, that is, the solution to the Volterra integral equation with measure-valued kernel
\begin{equation}\label{eq:r_Thm}
	r=-\mu_K-r*\mu_K.
\end{equation}
We denote $\Gamma_K(\rho):=\hat{\rho}$ and call $\Gamma_K$ the Volterra-Skorokhod map with kernel $K$.
\end{theorem}
\begin{proof}
The strategy is to convolve Condition~1 of Definition~\ref{def:VSP} with the resolvent $r$ of $-\mu_K$. This absorbs the memory term $\eta*\mu_K$ and exhibits $\hat\rho$ as the ordinary Skorokhod reflection of an explicit, $\hat\rho$-dependent input $\psi_\rho^K(\hat\rho)$. Then a Bielecki-norm (Section~\ref{sec:bielecki}) contraction provides the resulting fixed point. Reversing the computation yields existence. 

\emph{Step 1: Reduction to a Skorokhod fixed point.} Suppose $(\hat{\rho},\eta)$ solves the Volterra-Skorokhod problem for $\rho$. For $t\in[0,T]$, integration by parts of the Lebesgue-Stieltjes integral, using $K(0)=1$, gives
\begin{equation}\label{eq:ibp}
	\int_0^t K(t-u) d \eta(u)=\eta(t)+\int_0^t \eta(t-u) d \mu_K(u)=(\eta+\eta*\mu_K)(t),
\end{equation}
and Condition~1 of Definition~\ref{def:VSP} then gives
\begin{equation}\label{eq:hat_rho}
	\hat{\rho}=\rho-\eta-\eta*\mu_K.
\end{equation}
By \eqref{eq:r_Thm}, the resolvent $r$ satisfies
\begin{equation}\label{eq:r}
	r+r*\mu_K=-\mu_K,
\end{equation}
which exists, is unique, and is a signed measure of locally finite total variation by \cite[Thm 1.5, Chap 4]{Gripenberg1990}. Using $\eta * \mu_K = \mu_K * \eta$ (Section~\ref{sec:conv}) and the associativity $r*(\mu_K*\eta)=(r*\mu_K)*\eta$ \cite[Chapter 4, Theorem 1.2]{Gripenberg1990}, convolving \eqref{eq:hat_rho} by $r$ and applying \eqref{eq:r}, we have
\begin{equation}\label{eq:r*hat_rho}
	r*\hat{\rho}=r*\rho-(r+r*\mu_K)*\eta=r*\rho+\mu_K*\eta.
\end{equation}
Since \eqref{eq:hat_rho} gives $\eta*\mu_K=\rho-\eta-\hat{\rho}$, substituting into \eqref{eq:r*hat_rho} and rearranging,
\begin{equation}\label{eq:hat_rho_psi}
	\hat{\rho}=r*\rho+\rho-r*\hat{\rho}-\eta=\psi_{\rho}^{K}(\hat{\rho})-\eta,
\end{equation}
with $\psi_\rho^K$ as in \eqref{eq:psi_Thm}. As $\mu_K$ has no atom at $0$, we get $r(\{0\})=0$ by evaluating \eqref{eq:r} at $\{0\}$ and using $(r*\mu_K)(\{0\})=r(\{0\})\mu_K(\{0\})$. Hence $\psi_\rho^K(f)(0)=\rho(0)\leq1$ for every $f$. Therefore the reduced input inherits the admissible initial condition. Thus $(\hat\rho,\eta)$ solves the ordinary Skorokhod problem for $\psi_\rho^K(\hat\rho)$, and $\hat\rho$ is a fixed point of $\Gamma(\psi_\rho^K(\cdot))$.

\emph{Step 2: Uniqueness via contraction.} 
By the explicit formula for the classical Skorokhod regulator after Definition~\ref{def:SP}, for $\psi_1,\psi_2\in\cd([0,T])$ with $\psi_i(0)\leq1$ and $t\in[0,T]$,
\begin{equation*}
	|\Gamma(\psi_1)(t)-\Gamma(\psi_2)(t)|\leq|\psi_1(t)-\psi_2(t)|+\sup_{s\leq t}|\psi_1(s)-\psi_2(s)|.
\end{equation*}
Multiplying by $e^{-\lambda t}$ and using $e^{-\lambda t}\leq e^{-\lambda s}$ for $s\leq t$ gives $\|\Gamma(\psi_1)-\Gamma(\psi_2)\|_{\lambda,T}\leq L\|\psi_1-\psi_2\|_{\lambda,T}$ with $L=2$, for every $\lambda\geq0$. By Lemma~\ref{lem:conv-cadlag}, $\psi_\rho^K$ maps $\cd([0,T])$ into itself, and hence so does $\Gamma(\psi_\rho^K(\cdot))$.
Hence for $f_1,f_2\in\cd([0,T])$,
\begin{equation}\label{eq:contrac1}
	\|\Gamma(\psi_{\rho}^{K}(f_1))-\Gamma(\psi_{\rho}^{K}(f_2))\|_{\lambda,T}\leq L\|\psi_{\rho}^{K}(f_1)-\psi_{\rho}^{K}(f_2)\|_{\lambda,T}=L\|r*(f_1-f_2)\|_{\lambda,T}.
\end{equation}
Since $r(\{0\})=0$ and $r\in M_{\mathrm{loc}}([0,T])$, Lemma~\ref{lem:w-norm-bd} applied to \eqref{eq:contrac1} gives
\begin{equation*}
	\|\Gamma(\psi_{\rho}^{K}(f_1))-\Gamma(\psi_{\rho}^{K}(f_2))\|_{\lambda,T}\leq LR_r(\lambda)\|f_1-f_2\|_{\lambda,T},
\end{equation*}
where $\lim_{\lambda\rightarrow\infty}R_r(\lambda)=0$. Choosing $\lambda^*$ with $LR_r(\lambda^*)<1$ makes $\Gamma(\psi_\rho^K(\cdot))$ a contraction on the complete space $(\cd([0,T]),\|\cdot\|_{\lambda^*,T})$. Hence, by Banach's fixed point theorem $\hat\rho$ is its unique fixed point. Together with the unique solvability of the ordinary Skorokhod problem, $(\hat\rho,\eta)$ is the unique Volterra-Skorokhod solution.

\emph{Step 3: Existence.} Given $\rho$ and $K$, let $\hat\rho$ be the fixed point of $\Gamma(\psi_\rho^K(\cdot))$ from Step~2 and $\eta:=\psi_\rho^K(\hat\rho)-\hat\rho$ the associated Skorokhod regulator. Then \eqref{eq:hat_rho_psi} holds by construction. The computation of Step~1 is reversible: solving \eqref{eq:hat_rho_psi} for $\eta$ gives $\eta=(\rho-\hat\rho)+r*(\rho-\hat\rho)$, and convolving by $\mu_K$ with the help of  \eqref{eq:r} yields $\mu_K*\eta=-r*(\rho-\hat\rho)$. Adding these two equations returns \eqref{eq:hat_rho}. By the integration-by-parts identity~\eqref{eq:ibp}, $(\hat\rho,\eta)$ then satisfies Conditions~1 to 4 of Definition~\ref{def:VSP}, which proves existence.
\end{proof}
\begin{remark}\label{rem:fin-to-inf-horizon}
By uniqueness, the restriction to $[0,T]$ of the solution on $[0,T']$, $T<T'$, is the solution on $[0,T]$. The solutions are therefore consistent in $T$ and define the unique solution on $[0,\infty)$, which we also denote by $(\Gamma_K(\rho),\eta)$.
\end{remark}
\begin{corollary}\label{cor:eta-resolvent}
Under the hypotheses of Theorem~\ref{thm:V-S-sol}, the regulator $\eta$ and the reflected occupancy $\hat\rho$ are conjugate through the resolvent:
\begin{equation*}
	\rho-\hat\rho=\eta+\mu_K*\eta,\qquad \eta=(\rho-\hat\rho)+r*(\rho-\hat\rho).
\end{equation*}
Hence, given the resolvent $r$ of \eqref{eq:r_Thm}, each of $\eta$ and $\hat\rho$ determines the other by a bounded convolution. 
\end{corollary}
\begin{proof}
The first identity is \eqref{eq:hat_rho}. Solving \eqref{eq:hat_rho_psi} for $\eta$ gives the second, $\eta=\psi_\rho^K(\hat\rho)-\hat\rho=(\rho-\hat\rho)+r*(\rho-\hat\rho)$.
\end{proof}
\begin{remark}\label{rem:ac-resolvent}
Suppose in addition that $K$ is absolutely continuous, and let $k:=K'$. Then $\mu_K(ds)=k(s)\,ds$ and $r(ds)=q(s)\,ds$, where $q\in L^1_{\mathrm{loc}}(\R_+)$ is the unique solution of $q=-k-q*k$ \cite[Chapter 2, Theorem 3.1]{Gripenberg1990}. Indeed, $q(s)\,ds$ solves \eqref{eq:r_Thm}, and the resolvent is unique. For $K=1-G$ with service density $g$, this equation reads $q=g+g*q$. Hence $q=\sum_{n\geq1}g^{*n}$ is the renewal density of the service distribution.
\end{remark}

\begin{remark}
Like the classical Skorokhod map, $\Gamma_K$ is idempotent. Indeed, $\Gamma_K(\rho)\leq1$ and $\Gamma_K(\rho)(0)=\rho(0)$. Hence $(\Gamma_K(\rho),0)$ solves the problem for the input $\Gamma_K(\rho)$, and uniqueness gives $\Gamma_K(\Gamma_K(\rho))=\Gamma_K(\rho)$.
\end{remark}

\begin{proposition}\label{prop:lip}
For a given kernel $K$, the Volterra-Skorokhod map $\Gamma_K$ is Lipschitz continuous. Specifically, there exists a constant $L_{K,T}>0$ such that for $\rho_1, \rho_2\in\mathcal{D}([0, T], \mathbb{R})$
\begin{equation*}
	\|\Gamma_K(\rho_1)-\Gamma_K(\rho_2)\|_T\leq L_{K,T}\|\rho_1-\rho_2\|_T.
\end{equation*}
\end{proposition}
\begin{proof}
Recall that $\psi_{\rho}^{K}$ and $r$ are defined as in \eqref{eq:psi_Thm}-\eqref{eq:r_Thm}.  
By Step~2 of the proof of Theorem~\ref{thm:V-S-sol}, $\Gamma$ is $L$-Lipschitz in $\|\cdot\|_{\lambda,T}$ for every $\lambda\geq0$. Hence,
by Theorem~\ref{thm:V-S-sol}, for $\rho_1,\rho_2\in\cd([0,T])$ we have
\begin{align}\label{eq:Gk1}
\|\Gamma_K(\rho_1)-\Gamma_K(\rho_2)&\|_{\lambda,T}=\|\Gamma(\psi_{\rho_1}^{K}(\Gamma_K(\rho_1)))-\Gamma(\psi_{\rho_2}^{K}(\Gamma_K(\rho_2)))\|_{\lambda,T}\nonumber\\
&\leq L\|\psi_{\rho_1}^{K}(\Gamma_K(\rho_1))-\psi_{\rho_2}^{K}(\Gamma_K(\rho_2))\|_{\lambda,T},
\end{align}
Expanding the right-hand side of \eqref{eq:Gk1} using the formula in relation~\eqref{eq:psi_Thm} we have
\begin{align}\label{eq:Gk2}
{\|\Gamma_K(\rho_1)-\Gamma_K(\rho_2) \|}_{\lambda, T} &\leq L\|r*(\rho_1-\rho_2)+(\rho_1-\rho_2)-r*(\Gamma_K(\rho_1)-\Gamma_K(\rho_2))\|_{\lambda,T}\nonumber\\
&\leq L\|r*(\rho_1-\rho_2)\|_{\lambda,T}+L\|\rho_1-\rho_2\|_{\lambda,T}+L\|r*(\Gamma_K(\rho_1)-\Gamma_K(\rho_2))\|_{\lambda,T}.
\end{align}
By applying Lemma~\ref{lem:w-norm-bd} to \eqref{eq:Gk2} we have
\begin{equation}\label{eq:Gk1-Gk2_2}
\|\Gamma_K(\rho_1)-\Gamma_K(\rho_2)\|_{\lambda,T}\leq L(1+R_r(\lambda))\|\rho_1-\rho_2\|_{\lambda,T}+LR_r(\lambda)\|\Gamma_K(\rho_1)-\Gamma_K(\rho_2)\|_{\lambda,T}.
\end{equation}
Since $\lim _{\lambda \rightarrow \infty} R_r(\lambda)=0$, we can choose a sufficiently large $\lambda_0>0$ such that $LR_r(\lambda_0)\leq1/2$. Therefore, by \eqref{eq:Gk1-Gk2_2} we have
\begin{equation*}
\|\Gamma_K(\rho_1)-\Gamma_K(\rho_2)\|_{\lambda_0,T}\leq L(1+R_r(\lambda_0))\|\rho_1-\rho_2\|_{\lambda_0,T}+\frac{1}{2}\|\Gamma_K(\rho_1)-\Gamma_K(\rho_2)\|_{\lambda_0,T}.
\end{equation*}
Consequently
\begin{equation}\label{eq:Lip-wt-norm}
	\|\Gamma_K(\rho_1)-\Gamma_K(\rho_2)\|_{\lambda_0,T}\leq 2L(1+R_r(\lambda_0))\|\rho_1-\rho_2\|_{\lambda_0,T}.
\end{equation}
By \eqref{eq:w-norm-eq} and \eqref{eq:Lip-wt-norm}, we get
\begin{align}\label{eq:Gk1-Gk2_3}
e^{-\lambda_0 T}\|\Gamma_K(\rho_1)-\Gamma_K(\rho_2)\|_T\leq\|\Gamma_K(\rho_1)-\Gamma_K(\rho_2)\|_{\lambda_0,T}&\leq 2L(1+R_r(\lambda_0))\|\rho_1-\rho_2\|_{\lambda_0,T}\nonumber\\
&\leq2L(1+R_r(\lambda_0))\|\rho_1-\rho_2\|_T.
\end{align}
Multiplying both sides of \eqref{eq:Gk1-Gk2_3} by $e^{\lambda_0 T}$ we get the desired result
\begin{equation*}
		\|\Gamma_K(\rho_1)-\Gamma_K(\rho_2)\|_T\leq L_{K,T}\|\rho_1-\rho_2\|_T.
\end{equation*}
Here 
$$L_{K,T}=2L(1+R_r(\lambda_0))e^{\lambda_0 T},$$
a constant which depends on $T$ and $r$, and hence on $K$.
\end{proof}
\begin{remark}
A natural question is whether the Lipschitz property also holds with respect to the $J_1$ and $M_1$ metrics. One direct way to deduce $J_1$-Lipschitz continuity from uniform Lipschitz continuity for a map $F$ is to verify its compatibility with time changes, namely
$$
F(x\circ \te)=F(x)\circ \te,
$$
where $\te$ is any strictly increasing continuous function of $[0,T]$ onto itself (cf. \cite[Corollary 1.6]{Kruk2007}). In the present setting, however, this compatibility generally fails for $\Gamma_K$ because the convolution term depends on the elapsed time $t-s$ through the kernel $K(t-s)$. Therefore, the $J_1$-Lipschitz property of $\Gamma_K$, if true, requires a separate argument. Since uniform Lipschitz continuity is sufficient for the purposes of this paper, we leave this question for future study.
\end{remark}
As a corollary to Proposition~\ref{prop:lip} we have that the limit of a sequence of kernel-reflected processes is the kernel-reflection of the limit.
\begin{corollary}\label{cor:cmt}
Suppose $\rho^n,\rho\in \cd([0,T])$ satisfy $\|\rho^n-\rho\|_T\rightarrow 0$. Then for any kernel $K$ as in Definition~\ref{def:VSP},
$$\|\Gamma_K(\rho^n)-\Gamma_K(\rho)\|_T\rightarrow 0.$$
\end{corollary}
\begin{proof}
This is immediate from the Lipschitz continuity of $\Gamma_K$ (Proposition~\ref{prop:lip}), since $\|\Gamma_K(\rho^n)-\Gamma_K(\rho)\|_T\leq L_{K,T}\|\rho^n-\rho\|_T$.
\end{proof}

The next proposition records how $\Gamma_K$ acts on random inputs, together with the moments and regularity that a random input passes to its reflection.
 
\begin{proposition}\label{prop:reflected-process}
Fix a kernel $K$ as in Definition~\ref{def:VSP} and $T>0$. Then the following are true.
\begin{enumerate}[label=(\roman*), wide, labelindent = 0pt]
\item \emph{(Lipschitz continuity)} The maps $\rho\mapsto\Gamma_K(\rho)$ and $\rho\mapsto\eta$ are Lipschitz in ${\|\cdot\|}_T$.
\item \emph{(Nonanticipation)} For every $t\in[0,T]$, the restrictions $\Gamma_K(\rho)|_{[0,t]}$ and $\eta|_{[0,t]}$ depend on $\rho$ only through $\rho|_{[0,t]}$.
\item \emph{(Adaptedness)} If $\rho$ is an adapted c\`adl\`ag process on a filtered probability space $(\Omega,\cf,(\cf_t)_{t\in[0,T]},\bp)$, then $\hat\rho=\Gamma_K(\rho)$ and $\eta$ are adapted c\`adl\`ag processes, and $(\hat\rho,\eta)$ solves the Volterra-Skorokhod problem for $\rho$ path by path.
\item \emph{(Inheritance)} There is a constant $C_{K,T}$ with $\|\hat\rho\|_T\leq C_{K,T}\|\rho\|_T$ pathwise. Hence $\E\|\hat\rho\|_T^p\leq C_{K,T}^p \E\|\rho\|_T^p$ for every $p\geq1$. If, moreover $K$ is continuous and $\rho$ has continuous paths almost surely, then so do $\hat\rho$ and $\eta$.
\end{enumerate}
\end{proposition}
\begin{proof} 
(i) $\Gamma_K$ is Lipschitz by Proposition~\ref{prop:lip}. By Corollary~\ref{cor:eta-resolvent}, $\eta=(\rho-\hat\rho)+r*(\rho-\hat\rho)$. Convolution by the fixed locally finite measure $r$ is bounded on $(\cd([0,T]),\|\cdot\|_T)$ by \eqref{eq:w-norm-eq} and \eqref{eq:w-norm-bd}. Hence composing with the Lipschitz $\Gamma_K$ makes $\rho\mapsto\eta$ Lipschitz. 

(ii) 
By Remark~\ref{rem:fin-to-inf-horizon}, the restriction of $(\Gamma_K(\rho),\eta)$ to $[0,t]$ is the unique solution on $[0,t]$. Conditions~1 to 4 of Definition~\ref{def:VSP} on $[0,t]$ involve $\rho$ only through $\rho|_{[0,t]}$, which proves (ii).

(iii) 
Let $f_0:=\rho$ and $f_{k+1}:=\Gamma(\psi_\rho^K(f_k))$, pathwise. By Step~2 of the proof of Theorem~\ref{thm:V-S-sol}, $f_k\to\hat\rho$ uniformly on $[0,T]$ for every $\omega$. The classical Skorokhod map $\Gamma$ sends adapted c\`adl\`ag processes to adapted c\`adl\`ag processes. It therefore suffices to show that $\psi_\rho^K$ does the same. If $f$ is adapted and c\`adl\`ag, then $f$ and $\rho$ are progressively measurable, and Fubini's theorem shows that $(r*f)(t)=\int_0^t f(t-s)\,r(ds)$ and $(r*\rho)(t)$ are $\cf_t$-measurable. Hence $\psi_\rho^K(f)$ is adapted, and c\`adl\`ag by Lemma~\ref{lem:conv-cadlag}. By induction every $f_k$ is adapted, and hence so are $\hat\rho$ and $\eta=\psi_\rho^K(\hat\rho)-\hat\rho$. The identities of Definition~\ref{def:VSP} hold for every fixed $\omega$.

(iv) Since the reflection is one-sided, Condition 2 does not control $\hat\rho$ from below, and we use the Lipschitz bound instead. 
The pair $(0,0)$ solves the Volterra-Skorokhod problem for $\rho\equiv0$, and hence $\Gamma_K(0)=0$. By Proposition~\ref{prop:lip}, $\|\hat\rho\|_T=\|\Gamma_K(\rho)-\Gamma_K(0)\|_T\leq L_{K,T}\|\rho\|_T$. The moment bound follows on taking expectations. 

For continuity, fix $\omega$ with $\rho$ continuous and suppose $\Delta\eta(t)>0$ for some $t$. Thus the atom of $\eta$ at $t$ forces $\hat\rho(t)=1$ by the complementarity Condition~4. However Condition~1 with $K(0)=1$, continuity of $\rho$, and continuity of $K$ imply $\Delta\hat\rho(t)=-\Delta\eta(t)$, whence $\hat\rho(t^-)=1+\Delta\eta(t)>1$, contradicting Condition~2. Thus $\eta$ is continuous. Then $t\mapsto\int_0^tK(t-u)\,d\eta(u)$ is continuous ($\eta$ continuous, $K$ of bounded variation) as well and by Condition~1 so is $\hat\rho$.
\end{proof}

The next lemma shows that the resolvent is nonnegative when $K$ is nonincreasing.

\begin{lemma}\label{lem:r-positive}
Suppose, in addition to the hypotheses of Theorem~\ref{thm:V-S-sol}, that $K$ is nonincreasing. Then $\mu_K$ is a nonpositive measure and the resolvent $r$ of \eqref{eq:r_Thm} is the nonnegative measure
\begin{equation*}
	r=\sum_{n=1}^\infty(-\mu_K)^{*n}.
\end{equation*}
\end{lemma}
\begin{proof}
Write $a:=-\mu_K$. Since $K$ is nonincreasing, $\mu_K((s,t])=K(t)-K(s)\leq0$, and by definition $\mu_K(\{0\})=0$. Thus $a$ is a nonnegative measure with $a(\{0\})=0$ and $a((0,t])=1-K(t)$. The resolvent equation \eqref{eq:r_Thm} reads $r=a+a*r$, which is solved by the series $\sum_{n\geq1}a^{*n}$, where each $a^{*n}$ is nonnegative. 
To prove convergence on $[0,T]$, 
choose $\beta > 0$ with $q:= \int_{(0,T]} e^{-\beta s} a(ds) < 1$. This is possible by dominated convergence theorem since $a(\{0\})=0$.
On $\{s_1+\cdots+s_n\leq T\}$ we have $s_i \leq T$ for all $i$ and $1 \leq e^{\beta T} \prod_i e^{-\beta s_i}$.
Hence $a^{*n}([0,T])\leq e^{\beta T}  q^n$ and $\sum_{n\geq1}a^{*n}([0,T]) \leq e^{\beta T} q/(1-q)$.
Finally, the uniqueness of the resolvent \cite[Chapter 4, Theorem~1.5]{Gripenberg1990} identifies it with $r$.
\end{proof}
The Volterra-Skorokhod map is not monotone in general, as the following example shows.
\begin{example}\label{prop:non-mono}
There exists a kernel $K$, inputs $\rho_1\leq\rho_2$ in $\cd([0,\infty),\R)$, and a time $t$ such that $\Gamma_K(\rho_1)(t)>\Gamma_K(\rho_2)(t)$. 

Consider a toy deterministic queueing loss system with unit capacity and unit service times. Take the kernel to be $K=\1_{[0,1)}$, the survival function of the deterministic service time $1$. Let $\Lambda_1=\1_{[0.5,\infty)}$ and $\Lambda_2=\1_{[0.25,\infty)}+\1_{[0.5,\infty)}$ be two cumulative arrival functions. By \eqref{eq:mu_F}, $d\Lambda_1=\der_{0.5}$ and $d\Lambda_2=\der_{0.25}+\der_{0.5}$. Let $\rho_i(t)=\int_0^tK(t-u)\,d\Lambda_i(u)$ denote the corresponding unconstrained occupancy as a function of time.  This gives $\rho_1=\1_{[0.5,1.5)}$ and $\rho_2=\1_{[0.25,1.25)}+\1_{[0.5,1.5)}$. Hence $\rho_2-\rho_1=\1_{[0.25,1.25)}\geq0$.

We now derive the constrained occupancy. Since $\rho_1\leq1$, no reflection is needed, and $\eta_1 \equiv 0$ with $\Gamma_K(\rho_1)=\rho_1=\1_{[0.5,1.5)}$. For $\rho_2$, the pair $(\hat\rho_2,\eta_2)=(\1_{[0.25,1.25)},\1_{[0.5,\infty)})$ satisfies Definition~\ref{def:VSP}. Indeed, $\int_0^tK(t-u)\,d\eta_2(u)=K(t-0.5)\1_{\{t\geq0.5\}}=\1_{[0.5,1.5)}(t)$, and hence $\hat\rho_2=\rho_2-\1_{[0.5,1.5)}=\1_{[0.25,1.25)}\leq1$. Furthermore $\eta_2$ has a single atom at $0.5$, where $\hat\rho_2=1$, and thus the complementarity condition holds. By the uniqueness in Theorem~\ref{thm:V-S-sol}, $\Gamma_K(\rho_2)=\1_{[0.25,1.25)}$. Hence for $t\in[1.25,1.5)$ we have $\Gamma_K(\rho_1)(t)=1>0=\Gamma_K(\rho_2)(t)$, while $\rho_1\leq\rho_2$.
\end{example}

\begin{remark}
For the kernel $K=\1_{[0,1)}$ above, $\mu_K=-\delta_1$. Hence $a=-\mu_K=\delta_1$ and the resolvent of Lemma~\ref{lem:r-positive} is $r=\sum_{n\geq1}\delta_n\geq0$. Nonnegativity of the resolvent therefore does not imply monotonicity of $\Gamma_K$. 
Monotonicity fails because input blocked at the barrier leaves the system, whereas the same input offered to a less loaded system is admitted and occupies it over the full kernel support. Thus a smaller input can yield a larger reflected value at later times.
\end{remark}

The reflected process is thus not monotone. 
The regulator, however, is monotone for nonincreasing $K$, because it is then the smallest admissible regulator.
Let 
\beq\label{eq:S(rho)}
S(\rho):=\lcl \xi\in\cd([0,\infty),\R):\xi\text{ nondecreasing},\xi(0)=0,\int_0^tK(t-u)\,d\xi(u)\geq\rho(t)-1\text{ for all}\, t\geq 0 \rcl.
\eeq
denote the set of nondecreasing regulators starting at $0$ that constrain the process $\rho$ by \eqref{eq:VS-reg}, that is those satisfying Conditions~1 to 3 of Definition~\ref{def:VSP}. The following result proves that thet regulator $\eta$, that solves the Volterra-Skorokhod problem in~Definition~\ref{def:VSP} is the minimal element of $S(\rho)$. 
\begin{theorem}\label{prop:eta-min}
Let $K$ be as in Definition~\ref{def:VSP} and nonincreasing.
For $\rho\in\cd([0,\infty),\R)$, let $(\hat\rho,\eta)$ be the solution to the Volterra-Skorokhod problem of Definition~\ref{def:VSP}.
Then $\eta\in S(\rho)$, and $\eta(t)\leq\xi(t)$ for every $\xi\in S(\rho)$ and every $t\geq0$.    
\end{theorem}

\begin{proof}
Write $a:=-\mu_K$. As in the proof of Lemma~\ref{lem:r-positive}, $a$ is a nonnegative measure, and
\beq\label{eq:K-a}
a(\{0\})=0,\qquad a((0,\theta])=1-K(\theta),\qquad K(\theta)=1-a((0,\theta]),\qquad \theta\geq0.
\eeq
Note that the solution $\eta\in S(\rho)$. Indeed, from Definition~\ref{def:VSP}~Condition~1, we have 
$$
\int_0^tK(t-u)\,d\eta(u)=\rho(t)-\hat\rho(t) \geq\rho(t)-1,
$$ 
where the last inequality follows by Condition~2. Condition~3 supplies that $\eta$ is nondecreasing with $\eta(0)=0$. In particular we have $S(\rho)\neq\emptyset$.

Fix $\xi\in S(\rho)$ and set
\beq\label{eq:D+v}
\zeta:=\xi-\eta,\qquad v(t):=\int_0^tK(t-u)\,d\zeta(u).
\eeq
As a difference of nondecreasing c\`adl\`ag functions vanishing at the origin, $\zeta$ is c\`adl\`ag, of bounded variation on compacts, with $\zeta(0)=0$ and $\mu_\zeta(\{0\})=0$. We show that $\zeta\geq0$ in the following steps.

\noindent
\emph{Step 1: Renewal identity for $\zeta$.} 
Applying \eqref{eq:ibp} to $\xi$ and to $\eta$, and using $\mu_K=-a$, we obtain for every $t\geq0$
\beq\label{eq:ibp-D}
v(t)=\zeta(t)-\int_0^t \zeta(t-s)\,a(ds).
\eeq
Eq.~\eqref{eq:ibp-D} converts a statement about the convolved difference $v$ into one about the raw difference $\zeta$, at the expense of the memory term $\int_0^t \zeta(t-s)\,a(ds)$. However, the mass that $a$ can contribute over a window of length $h$, $a((0,h])=1-K(h)$, is small for small $h$. Below we use this fact to complete the proof in Step~5.

\noindent
\emph{Step 2: The two phases $\hat\rho =1$ and $\hat\rho < 1$.} By using Definition~\ref{def:VSP}~Condition~1 in \eqref{eq:D+v} we have for all $t \geq 0$
\begin{align}\label{eq:v-active}
v(t) &=  \int_0^t K(t-u)\,d\xi(u)- \int_0^t K(t-u)\,d\eta(u) \nonumber \\
&= \int_0^tK(t-u)\,d\xi(u)-\big(\rho(t)-\hat\rho(t)\big)\geq\hat\rho(t)-1,
\end{align}
where the last inequality follows from the feasibility of $\xi \in S(\rho)$.
By Definition~\ref{def:VSP}~Condition~2, the right side of \eqref{eq:v-active} is nonpositive, and it vanishes exactly on the contact set $B(\hat\rho) = \{t: \hat\rho(t)=1\}$. Therefore from \eqref{eq:v-active} we have
\beq\label{eq:v-contact}
v(s)\geq0\qquad\text{whenever }\hat\rho(s)=1.
\eeq
For times $t \not\in B(\hat\rho)$, the above provides no sign for $v(t)$. For this case, we use Definition~\ref{def:VSP}~Condition~4, which 
gives $\Delta\eta(t)=0$
whenever $\hat\rho(t)<1$, 
and $\eta(t) = \eta(\si)$ whenever $\hat{\rho} < 1$ on $(\si, t]$.
Hence
\beq\label{eq:D-frozen}
\hat\rho<1\text{ on }(\sigma,t]\ \Longrightarrow\ \zeta(t)-\zeta(\sigma)=\xi(t)-\xi(\sigma)\geq0.
\eeq
The proof alternates between \eqref{eq:v-contact} and \eqref{eq:D-frozen}.

\noindent
\emph{Step 3: First passage.} Let $T^*:=\inf\{t\geq0:\zeta(t)<0\}$ and suppose, for contradiction, that $T^*<\infty$. Then
\beq\label{eq:pre-T}
\zeta\geq0\ \text{ on }[0,T^*), \quad \text{and hence} \quad \zeta(T^{*-})\geq0\ \qquad \text{ if }T^*>0.
\eeq
We claim $\zeta(T^*)\geq0$. If $T^*=0$, this follows from $\zeta(0)=0$. Let now $T^*>0$. 

\noindent
(i) If $\hat\rho(T^*)=1$, we evaluate \eqref{eq:ibp-D} at $t=T^*$ to get
$$
\zeta(T^*)=v(T^*)+\int_0^{T^*} \zeta(T^*-s)\,a(ds)\geq0,
$$
where the inequality follows from the fact that $v(T^*) \geq 0$ when $\hat\rho(T^*) = 1$, $\zeta(T^* -s) \geq 0$ for $s \in (0, T^*)$ from \eqref{eq:pre-T}, $\zeta(0)=0$ for $s=T^{\ast}$, and $a$ is a nonnegative measure.

\noindent
(ii) If instead $\hat\rho(T^*)<1$, then $\Delta\eta(T^*)=0$ by Condition~4. Therefore $\zeta(T^*)=\zeta(T^{*-})+\Delta\xi(T^*)\geq0$ 
by \eqref{eq:pre-T} 
and monotonicity of $\xi$.

\noindent
\emph{Step 4: Bound on a short window.} By \eqref{eq:K-a} and right continuity of $K$ at the origin, $a((0,h])=1-K(h)\to0$ as $h\downarrow0$. Choose an $h>0$ such that
\beq\label{eq:h-choice}
a\big((0,h]\big)<1.
\eeq
Being c\`adl\`ag, $\zeta$ is bounded on the compact interval $[T^*,T^*+h]$. Therefore
\beq\label{eq:m}
m:=\min\Big(0,\inf_{[T^*,T^*+h]}\zeta\Big)\in(-\infty,0].
\eeq
We claim
\beq\label{eq:star}
\zeta(t)\geq m\,a\big((0,h]\big)\qquad\text{for every }t\in[T^*,T^*+h].
\eeq

Consider first a contact time $s\in[T^*,T^*+h]$ where $\hat\rho(s)=1$. Splitting the integral in \eqref{eq:ibp-D} we have
\beq\label{eq:v-int-split}
\int_0^s \zeta(s-r) a(dr) = \int_{(0, s-T^*]} \zeta(s-r)a(dr) + \int_{(s-T^*, s)} \zeta(s-r)a(dr).
\eeq
The points $r=0$ and $r=s$ are omitted because $a(\{0\})=0$ and $\zeta(0)=0$.
For $r\in(s-T^*,s)$ the argument $s-r$ lies in $(0,T^*)$, where $\zeta\geq0$ by \eqref{eq:pre-T}.
For $r\in(0,s-T^*]$ it lies in $[T^*,s)\subseteq[T^*,T^*+h]$, where $\zeta\geq m$ by the definition of $m$ in \eqref{eq:m}. Since $a\geq0$, discarding the second integral and bounding the first in \eqref{eq:v-int-split}, along with \eqref{eq:v-contact} gives
\beq\label{eq:active-bd}
\zeta(s)=v(s) + \int_0^s \zeta(s-r) \, a(dr) \geq m a \lp (0,s-T^*] \rp \geq m a \lp (0,h] \rp,
\eeq
where the last inequality holds from  $a((0,s-T^*])\leq a((0,h])$ and the fact that $m\leq0$. For $s=T^*$ the range $(0,s-T^*]$ is empty and \eqref{eq:active-bd} reduces to $\zeta(T^*)\geq0$, consistent with Step~3.

Now fix $t\in[T^*,T^*+h]$. If $\hat\rho(t)=1$, then \eqref{eq:active-bd} at $s=t$ is \eqref{eq:star}. If $\hat\rho(t)<1$, let
$$
\sigma:=\sup\{s\in[T^*,t]:\hat\rho(s)=1\},\qquad\text{with the convention }\sup\emptyset:=T^*,
$$
the last contact time at or before $t$. By Definition~\ref{def:VSP}~Condition~2 and the definition of $\sigma$ we have $\hat\rho<1$ on $(\sigma,t]$. Therefore $\zeta(t)\geq \zeta(\sigma)$ by \eqref{eq:D-frozen} and it suffices to bound $\zeta(\sigma)$. Consider three cases:
\ben[label=(\roman*)]
\item $\sigma=T^*$. Then by Step 3, $\zeta(\sigma)\geq0\geq m\,a((0,h])$ since $m\leq0$.
\item $\sigma>T^*$ and $\hat\rho(\sigma)=1$. Then \eqref{eq:active-bd} applies at $s=\sigma$.
\item $\sigma>T^*$ and $\hat\rho(\sigma)<1$. Here the supremum is not attained which may occur if $\hat\rho$ jumps down at $\sigma$ from the barrier. Choose $s_n\uparrow\sigma$ with $\hat\rho(s_n)=1$ such that each $s_n$ lies in $[T^*,T^*+h]$.  Thus \eqref{eq:active-bd} gives $\zeta(s_n)\geq m\,a((0,h])$, and letting $n\to\infty$ yields $\zeta(\sigma^-)\geq m\,a((0,h])$. By Condition~4, $\Delta\eta(\sigma)=0$, and hence $\zeta(\sigma)=\zeta(\sigma^-)+\Delta\xi(\sigma)\geq \zeta(\sigma^-)\geq m\,a((0,h])$.
\een
In every case $\zeta(t)\geq \zeta(\sigma)\geq m\,a((0,h])$. This proves \eqref{eq:star}.

\emph{Step 5: Contradiction.} Taking the infimum of \eqref{eq:star} over $[T^*,T^*+h]$ and using $m a((0,h])\leq0$, we have
\beq\label{eq:contra-1}
m=\min\Big(0,\inf_{[T^*,T^*+h]}\zeta\Big)\geq m\,a\big((0,h]\big),\qquad\text{i.e.}\qquad m\big(1-a((0,h])\big)\geq0.
\eeq
Since $1-a((0,h])$ is strictly positive by \eqref{eq:h-choice}, relation~\eqref{eq:contra-1} implies $m\geq0$. But $m\leq0$ from \eqref{eq:m} forces $m=0$, and therefore $\zeta\geq0$ on $[T^*,T^*+h]$. Combined with \eqref{eq:pre-T}, we have $\zeta\geq0$ on $[0,T^*+h]$. This contradicts the definition of $T^*$. Hence $T^*=\infty$, i.e. $\zeta(t)\geq0$ or equivalently $\eta(t)\leq\xi(t)$ for all $t \geq 0$.
\end{proof}

Minimality yields monotonicity of the regulator.
\begin{corollary}
\label{thm:eta-mono}
Under the hypotheses of Theorem~\ref{prop:eta-min}, let $\rho_1,\rho_2\in\cd([0,\infty),\R)$ with $\rho_1\leq\rho_2$ pointwise, and let $\eta_1,\eta_2$ be the corresponding regulators. Then $\eta_1(t)\leq\eta_2(t)$ for all $t\geq0$.
\end{corollary}
\begin{proof}

Recall the notation $S(\rho)$ introduced in relation~\eqref{eq:S(rho)}. Let $\xi \in S(\rho_2)$. Then $\int_0^tK(t-u)\,d\xi(u)\geq\rho_2(t)-1$. Since $\rho_1 \leq \rho_2$, this implies
$\int_0^tK(t-u)\,d\xi(u)\geq\rho_1(t)-1$, and hence $\xi \in S(\rho_1)$. Thus $S(\rho_2)\subseteq S(\rho_1)$. By Theorem~\ref{prop:eta-min}, $\eta_2\in S(\rho_2)\subseteq S(\rho_1)$ whereas $\eta_1$ is the least element of $S(\rho_1)$. Hence $\eta_1\leq\eta_2$.
\end{proof}
The results above are pathwise. They transfer to stochastic process inputs through couplings.
\begin{corollary}\label{cor:eta-st}
Suppose that $K$ satisfies the hypotheses of Theorem~\ref{prop:eta-min}.
Let $\rho_1,\rho_2$ be $\cd([0,\infty),\R)$-valued stochastic processes, admitting a coupling with $\rho_1\leq\rho_2$ pointwise almost surely. Then their regulators admit a coupling with $\eta_1\leq\eta_2$ pointwise almost surely. In particular, $\eta_1(t)$ is stochastically smaller than $\eta_2(t)$ at every $t\geq0$, and $\E[\varphi(\eta_1)]\leq\E[\varphi(\eta_2)]$ for every bounded measurable $\varphi$ that is nondecreasing for the pointwise order.
\end{corollary}
\begin{proof}
On the coupling, Corollary~\ref{thm:eta-mono} gives $\eta_1\leq\eta_2$ pointwise almost surely, and the stated consequences follow.
\end{proof}

\section{Fluid Limits of Loss Queues}\label{sec:queue}
\subsection{Model and fluid limit}\label{sec:queue-model}
\begin{assumption}\label{asm:queue0}
	Consider a $G_t / G / n / n$ loss queueing system, namely, a queueing system with
	\begin{enumerate}[ wide, labelindent = 0pt, label=\roman*.]
		\item an arrival process $A^n$ with $A^n(0)=0$ and arrival epochs $T_1\leq T_2\leq\cdots$, whose dependence on $n$ we suppress, satisfying a functional strong law of large numbers. More precisely, there exists a continuous nondecreasing function $\Lambda$ with $\Lambda(0)=0$, such that, for every $T > 0$,
		\begin{equation}\label{eq:flln-A}
			\lim_{n\rightarrow \infty}\left\|\bar{A}^n-\Lambda\right\|_T=0,\text{ a.s.}
		\end{equation}	
		where $\bar{A}^n=A^n/n$,
		\item general customer service times $\{S_i\}_{i\geq1}$, with $S_i$ the service time of the $i$-th arriving customer, sampled independently from a distribution $G$, with $G(0)=0$, independent of the arrival process $A^n$.
		\item the system has $n$ servers and zero buffer or waiting space. That is, when all $n$ servers are busy, incoming customer arrivals are lost. For simplicity the system is considered empty at time $0$.
		\item \emph{(no ties)} almost surely the arrival epochs $\{T_i\}$ are distinct and no service completion coincides with an arrival epoch: $T_i\neq T_j+S_j$ for all $i\neq j$.
	\end{enumerate}
All systems, $n\geq1$, are defined on a common probability space $(\Omega,\cf,\bp)$, and the service sequence $\{S_i\}$ does not depend on $n$.    
\end{assumption}

\begin{remark}
Assumption~\ref{asm:queue0}(iv) holds almost surely in the standard settings. Distinct arrival epochs ensure $A^n$ has jumps of size $1$. The condition $T_i\neq T_j+S_j$ holds almost surely if $G$ is continuous, or if the differences $T_i-T_j$, $i\neq j$, have atomless laws and the arrival process is independent of $\{S_i\}$ (e.g. Poisson or renewal arrivals with continuous interarrival times, even for deterministic service). 
Since the arrival process is independent of $S_j$, $\bp(T_i=T_j+S_j)=\int\bp(T_i-T_j=s)\,G(ds)$, and $\bp(T_i-T_j=s)>0$ for at most countably many $s$. Hence the condition can fail only if $G$ has an atom at some $s$ with $\bp(T_i-T_j=s)>0$, as for deterministic service combined with periodic arrivals.
The assumption removes the ambiguity of a customer that finds all servers busy at the instant one is vacated. Under this assumption, the blocking rule $\{\bar N^n_{T_i-}=1\}$ and the occupancy value $\bar N^n_{T_i}=1$ agree, and the complementarity condition of the Volterra-Skorokhod problem holds.
\end{remark}
Consider the $G_t/G/n/n$ loss system as in Assumption~\ref{asm:queue0}. 
Let $N_t^n$ denote the number of occupied servers, or equivalently the number of customers in the system at time $t$. Let $\bar{N}^n_t:= \frac{N_t^n}{n}$ be the fraction of occupied servers or the $n$-scaled number of customers in the system at time $t$. 
The fraction of occupied servers satisfies
\begin{equation*}
	\bar{N}_t^n=\frac{1}{n}\sum_{i=1}^{\infty}  \1_{\{T_{i} \leq t\} } \1_{\{\bar{N}_{T_{i}-}^{n}<1\}} \1_{\{S_i >t-T_{i}\}} .
\end{equation*}
For the corresponding $G_t/G/\infty$ queue without blocking, the scaled number of customers in the system is
\begin{equation*}
\rho_t^n=\frac{1}{n}\sum_{i=1}^{\infty}  \1_{\{T_{i} \leq t\} }\1_{\{S_i >t-T_{i}\}}. 
\end{equation*} 
By the functional law of large numbers for the $G_t/G/\infty$ queue \cite{whitt2002,pangwhitt2010,krichagina_puhalskii1997},
\begin{equation}\label{eq:ggn}
	\rho_t^n\rightarrow\rho_t:=\int_0^t(1-G(t-u))d\Lambda(u), \text{ a.s.}
\end{equation} 
in the uniform topology.

The concentration argument below treats the service times of the blocked customers as an i.i.d. sample, independent of the blocked arrival epochs. The next lemma justifies this. A blocked customer is lost on arrival and occupies no server, and hence its service time never influences which customers are blocked.

\begin{lemma}
\label{lem:blocked-marks}
Under Assumption~\ref{asm:queue0}, fix $T > 0$, let $B^n=\{i \leq A^n(T):\bar N^n_{T_i-}=1\}$ be the set of blocked customers by time $T$, and let $\iota_1^n < \cdots < \iota_{|B^n|}^n$ be its elements. Write $\tau^n_j:=T_{\io^n_j}$ and $\sigma^n_j:=S_{\io^n_j}$ for the arrival epoch and the service time of the $j$-th blocked customer. Then, conditionally on the arrival process $A^n$ and on $B^n$, the service times $(\sigma_j^n)_{j \leq |B^n|}$ are i.i.d. with distribution $G$. 
\end{lemma}
\begin{proof}
Let $\beta_i=\1_{\{i\in B^n\}}$. By Assumption~\ref{asm:queue0}(iv), $T_1 < T_2< \cdots$ almost surely and no departure occurs at an arrival epoch. Hence the occupancy just before the $i$-th arrival is
\beq\label{eq:block-rule}
	\bar N^n_{T_i-}=\frac1n\sum_{k<i,\,\beta_k=0}\1_{\{S_k>T_i-T_k\}}.
\eeq
Thus whether customer $i$ is blocked is determined by the arrival process and the service times of previously admitted customers. 
We condition on $A^n$ and let $m:=A^n(T)$. For $b\in\{0,1\}^m$, define $\beta^b_1,\ldots,\beta^b_m$ by \eqref{eq:block-rule} with the sum restricted to the indices $k<i$ with $b_k=0$, that is,
$$
\beta^b_i:=\1_{\{\frac1n\sum_{k<i,\,b_k=0}\1_{\{S_k>T_i-T_k\}}=1\}}.
$$
By induction on $i$, $\{\beta=b\}=\{\beta^b=b\}$, and the latter event lies in $\si(A^n,S_k:b_k=0)$. Given $A^n$, the marks $(S_k)_{k:\,b_k=1}$ are i.i.d. with distribution $G$ and independent of $\si(A^n,S_k:b_k=0)$. Hence, on $\{\beta=b\}=\{B^n=\{k\leq m:b_k=1\}\}$, the blocked service times are i.i.d. with distribution $G$ given $A^n$ and $B^n$. Summing over $b$ proves the lemma. This is a form of the optional skipping of \cite{doob1936note}.
\end{proof}
\begin{theorem}\label{thm:fluid}
Suppose that Assumption~\ref{asm:queue0} holds. Fix $T>0$, and let $K=1-G$. Then $\|\bar{N}^n-\Gamma_K(\rho^n)\|_T\rightarrow 0$ a.s. Furthermore, $\|\bar{N}^n-\Gamma_K(\rho)\|_T\rightarrow 0$ a.s.
\end{theorem}
\begin{proof}
Let $\cm^n$ be the counting measure of the marked point process $(T_i,S_i)$ of the $n$-th system. Then the fraction of occupied servers of the loss model can be written in the integral form:
\begin{align}\label{eq:rep-N}
\bar{N}_t^n=\int_0^t\int_{\R}\frac{1}{n} \1_{\{x>t-u\}} \1_{\{\bar{N}_{u-}^n<1\}}d\cm^n(du,dx)&=\rho_t^n-\int_0^t\int_{\R} \frac{1}{n} \1_{\{x>t-u\}} \1_{\{\bar{N}_{u-}^n=1\}}d\cm^n(du,dx)\nonumber\\
&=\rho_t^n-\int_0^t \1_{\{S_{A^n(u)}>t-u\}}d\bar{\eta}^n(u).
\end{align}
where $\bar{\eta}^n(t)=\frac{1}{n}\sum_{i=1}^{A^n(t)}\1_{\{\bar{N}^n_{T_i-}=1\}}$. For simplicity we denote $K^n(t,u):=\1_{\{S_{A^n(u)}>t-u\}}$. Then,
\begin{align}\label{eq:rep-N-2}
	\bar{N}^n_t=\rho_t^n-\int_0^t K^n(t,u)d\bar{\eta}^n(u)&=\rho_t^n-\int_0^t\left(K^n(t,u)-K(t-u)\right)d\bar{\eta}^n(u)-\int_0^t K(t-u)d\bar{\eta}^n(u)\nonumber\\
	&=\tilde{\rho}^n_t-\int_0^t K(t-u)d\bar{\eta}^n(u),
\end{align}
where $\tilde{\rho}^n_t=\rho_t^n-\int_0^t\left(K^n(t,u)-K(t-u)\right)d\bar{\eta}^n(u)$. The pair $(\bar N^n,\bar\eta^n)$ solves the Volterra-Skorokhod problem with kernel $K$ for $\tilde\rho^n$. Indeed, Eq.~\eqref{eq:rep-N-2} is Condition~1 of Definition~\ref{def:VSP} with regulator $\bar\eta^n$, while Condition~2 holds since $\bar N^n\leq1$. Condition~3 holds since $\bar\eta^n$ is nondecreasing with $\bar\eta^n(0)=0$, since $A^n(0)=0$. For the complementarity Condition~4, $\bar\eta^n$ increases only at arrival epochs $T_i$ with $\bar N^n_{T_i-}=1$. At such an epoch, Assumption~\ref{asm:queue0}(iv) rules out a simultaneous departure, that is, $\bar N^n_{T_i}=\bar N^n_{T_i-}=1$ and therefore $\int_0^\infty\1_{\{\bar N^n_t<1\}}\,d\bar\eta^n(t)=0$. 
Since the system starts empty and $A^n(0)=0$, we have $\rho^n(0)=\tilde{\rho}^n(0)=0$.
By Proposition~\ref{prop:lip} we have
\begin{align*}
	\|\bar{N}^n-\Gamma_K(\rho^n)\|_T=\|\Gamma_K(\tilde{\rho}^n)-\Gamma_K(\rho^n)\|_T&\leq L_{K,T}\|\tilde{\rho}^n-\rho^n\|_T\\
	&=L_{K,T}\sup_{t \in[0, T]}\left|\int_0^t\left(K^n(t,u)-K(t-u)\right)d\bar{\eta}^n(u)\right|.
\end{align*}
Denote $E^n(t)=\int_0^t\left(K^n(t,u)-K(t-u)\right)d\bar{\eta}^n(u)$. We will show that $\|E^n\|_T\rightarrow0$ almost surely as $n\rightarrow\infty$. Using the representation in \eqref{eq:rep-N} we have
\begin{align}\label{eq:En}
	E^n(t)&=\frac{1}{n}\sum_{i=1}^{A^n(t)}\1_{\{\bar{N}_{T_i-}^n=1\}}\left( \1_{\{S_i>t-T_i\}} -(1-G(t-T_i))\right)\nonumber\\
	&=\frac{1}{n}\sum_{i=1}^{J^n(t)}\left( \1_{\{\sigma_i^n>t-\tau_i^n\}} -(1-G(t-\tau_i^n))\right),
\end{align}
where $J^n(t)=\sum_{i=1}^{A^n(t)}\1_{\{\bar{N}_{T_i-}^n=1\}}$ is the number of customers blocked by time $t$, and $\tau_j^n$, $\si_j^n$ are as in Lemma~\ref{lem:blocked-marks}. For any time $t \in [0, T]$, let $H_t \subset \mathbb{R}_+ \times \mathbb{R}_+$ be the open half-plane defined by:
$$H_t = \{(u, x): u + x > t\}.$$
The family of sets $\mathcal{H}_{T} = \{H_t: t \in [0, T]\}$ consists of half-planes bounded by diagonal lines of slope $-1$. Consequently, $\mathcal{H}_T$ is a Vapnik-Chervonenkis (VC) class (see for example \cite[Section 2.6]{vanderVaart1996}). By Lemma~\ref{lem:blocked-marks}, conditionally on $A^n$ and $B^n$, the marks $(\sigma_j^n)$ are i.i.d. with distribution $G$, and each $\tau^n_j=T_{\io^n_j}$ is a function of $A^n$ and $B^n$. Setting $P^n_i(\cdot)=\delta_{\tau_i^n}\otimes G$, the indicator $\1_{\{\sigma_i^n>t-\tau_i^n\}}=\1_{\{(\tau_i^n,\sigma_i^n)\in H_t\}}$ has conditional mean $P^n_i((\tau_i^n,\sigma_i^n)\in H_t)=1-G(t-\tau_i^n)=K(t-\tau_i^n)$, which is the centering in \eqref{eq:En}. Define
$$\Delta_m^n(\mathcal{H}_T) := \sup_{H \in \mathcal{H}_T} \left| \frac{1}{m} \sum_{i=1}^m \left( \1_{\{(\tau_i^n, \sigma_i^n) \in H\}} - P^n_i((\tau_i^n, \sigma_i^n) \in H) \right) \right|,$$
and set $\Delta_m^n(\mathcal H_T)=0 \text{ if }m=0\text{ or }m>J^n(T).$ By the symmetrization and combinatorial arguments of \cite[Section II.3]{Pollard1984}, which use only the independence of the summands, 
we have, for some positive constants $C,C^{\prime}$ and $c$
\begin{equation*}
	\bp\left(\Delta_m^n(\mathcal{H}_T) >\varepsilon \mid A^n, B^n \right)\leq CS_{\mathcal{H}_T}(m)\exp\left(-c\varepsilon^2m\right)\leq C^{\prime}m\exp\left(-c\varepsilon^2m\right),
\end{equation*}
where $S_{\ch_T}(m):=\max_{z_1,\ldots,z_m}\#\{H\cap\{z_1,\ldots,z_m\}:H\in\ch_T\}$ is the shatter coefficient of $\ch_T$, and where the last inequality follows from Sauer's lemma since $\ch_T$ is a VC class. Consequently, for any fixed $\delta>0$ and $ M<\infty$ 
\begin{equation*}
			\bp\left(\max _{\delta n \leq m \leq M n} \Delta_m^n\left(\mathcal{H}_T\right)>\varepsilon \mid A^n, B^n \right)  \leq \sum_{m=\lceil\delta n\rceil}^{\lfloor M n\rfloor} C^{\prime} m e^{-c \varepsilon^2 m} \leq C^{\prime} \sum_{m=\lceil\delta n\rceil}^{\infty} m e^{-c \varepsilon^2 m} \leq C_{\varepsilon, \delta} n e^{-c \varepsilon^2 \delta n / 2},
\end{equation*}
for large enough $n$. The right-hand side is summable in $n$. Taking expectations and applying the Borel-Cantelli lemma, we obtain, for any fixed $\delta>0$ and $ M<\infty$ 
\begin{equation}\label{eq:Hto0}
	\max _{\delta n \leq m \leq M n}\Delta_m^n(\mathcal{H}_T)\rightarrow 0, \text{ a.s.}
\end{equation}
By \eqref{eq:flln-A},
with the deterministic constant $M=\Lambda(T)+1$, there exists a random $N_1<\infty$ such that $A^n(T)\leq nM$ for all $n>N_1$. Fix $t\in[0,T]$. Since $H_t\in\mathcal{H}_T$, \eqref{eq:En} and the centering above give
\begin{equation}\label{eq:En-pt}
	|E^n(t)|=\frac{J^n(t)}{n}\left|\frac{1}{J^n(t)}\sum_{i=1}^{J^n(t)}\left(\1_{\{(\tau_i^n,\sigma_i^n)\in H_t\}}-P^n_i((\tau_i^n,\sigma_i^n)\in H_t)\right)\right|\leq\frac{J^n(t)}{n}\Delta^n_{J^n(t)}(\mathcal{H}_T),
\end{equation}
with the empty sum ($J^n(t)=0$) read as zero. Since $J^n(t)\in\{0,1,\dots,J^n(T)\}$, taking the supremum in \eqref{eq:En-pt} and using $J^n(T)\leq A^n(T)\leq nM$ for $n>N_1$,
\begin{equation}\label{eq:Enbd}
	\|E^n\|_T\leq \max_{1\leq m\leq J^n(T)}\frac{m}{n}\Delta_{m}^n(\mathcal{H}_T)\leq \max_{1\leq m\leq nM}\frac{m}{n}\Delta_{m}^n(\mathcal{H}_T).
\end{equation}
For any $\ep>0$ we choose $\delta<\ep$. By \eqref{eq:Hto0} we know that there exists large enough $N_2$ such that for $n>N_2,\max _{\delta n \leq m \leq M n} \Delta_{m}^n(\mathcal{H}_T)<\frac{\ep-\delta}{M}$. Let $N=\max(N_1,N_2)$. For all $n > N$, we have by \eqref{eq:Enbd}
\begin{equation*}
	\|E^n\|_T\leq \max_{1\leq m\leq \delta n}\frac{m}{n}\Delta_{m}^n(\mathcal{H}_T)+\max_{\delta n\leq m\leq nM}\frac{m}{n}\Delta_{m}^n(\mathcal{H}_T)< \frac{\delta n}{n}+M \frac{\ep-\delta}{M}=\ep.
\end{equation*}
Hence $\|E^n\|_T\rightarrow 0$ a.s., and consequently $\|\bar{N}^n-\Gamma_K(\rho^n)\|_T\rightarrow 0$ a.s. as well.

By \eqref{eq:ggn} and Corollary~\ref{cor:cmt} we have $\|\Gamma_K(\rho^n)-\Gamma_K(\rho)\|_T\rightarrow 0$. Therefore, by triangle inequality we have $\|\bar{N}^n-\Gamma_K(\rho)\|_T\rightarrow 0$ a.s.
\end{proof}

\begin{remark}\label{rk:approx}
By Theorem~\ref{thm:fluid}, the fluid limit of the $G_t/G/n/n$ queue is the Volterra-Skorokhod reflection of the fluid limit of the $G_t/G/\infty$ queue. The identity
\begin{equation*}
	\rho(t)=\Gamma_K(\rho)(t)+\int_0^t(1-G(t-u))\,d\eta(u)
\end{equation*}
splits the infinite-server queue into the loss queue and a queue of blocked customers with arrival process $\eta$. In the pre-limit, $\Gamma_K(\rho^n)$ differs from $\bar N^n$ because the regulation kernel $\1_{\{S_{A^n(u)}>t-u\}}$ in \eqref{eq:rep-N} depends on the service marks and is not a function of $t-u$ alone. Replacing it by $K(t-u)$, as in \eqref{eq:rep-N-2}, replaces the departures of blocked customers by their mean flow. The resulting error $E^n$ vanishes as $n\to\infty$.
\end{remark}

Recall $\bar{\eta}^n(t)=\frac{1}{n}\sum_{i=1}^{A^n(t)}\1_{\{\bar{N}^n_{T_i-}=1\}}$, the scaled number of blocked customers. The following result shows that the regulator of the limiting process is the fluid limit of $\bar{\eta}$ and inherits the order properties of Section~\ref{sec:vsp}.

\begin{corollary}\label{cor:loss-order}
Suppose that Assumption~\ref{asm:queue0} holds. Fix $T>0$, let $K=1-G$, and let $\eta$ be the regulator for the input $\rho$. Then $\|\bar\eta^n-\eta\|_T\to0$ a.s. Moreover, let $\Lambda_1,\Lambda_2$ be continuous nondecreasing functions with $\Lambda_i(0)=0$ and $d\Lambda_1\leq d\Lambda_2$. If $\rho_i:=K*d\Lambda_i$ as in \eqref{eq:ggn}, and $\eta_i$ is the regulator for the input $\rho_i$, then $\eta_1\leq\eta_2$.
\end{corollary}
\begin{proof}
By the proof of Theorem~\ref{thm:fluid} and \eqref{eq:ggn}, $\bar\eta^n$ is the regulator for the input $\tilde\rho^n$, and $\|\tilde\rho^n-\rho\|_T\leq\|E^n\|_T+\|\rho^n-\rho\|_T\to0$ a.s. The first claim follows from Proposition~\ref{prop:reflected-process}(i). For the second, $K=1-G\geq0$ gives $\rho_2-\rho_1=K*d(\Lambda_2-\Lambda_1)\geq0$. Since $G(0)=0$, $K$ is c\`adl\`ag, nonincreasing, of bounded variation, with $K(0)=1$, and Corollary~\ref{thm:eta-mono} gives $\eta_1\leq\eta_2$.
\end{proof}

A larger arrival rate thus produces a larger cumulative fluid loss. The fluid occupancy need not increase in the same way, as Example~\ref{prop:non-mono} shows for a loss system with unit service times.

\subsection{Thinning representation of the fluid limit}

The input of Theorem~\ref{thm:fluid} has the form $\rho=K*d\Lambda$, in which arrivals enter through the reflection kernel $K=1-G$. For such inputs the regulator is a thinning of the arrivals, with $w(t)$ the fraction of fluid arrivals admitted at time $t$. 
Theorem~\ref{prop:thinning} below proves this for arrival measures with atoms. Under Assumption~\ref{asm:queue0}, $\Lambda$ is continuous, and we include atoms because the result holds in this generality. Corollary~\ref{thm:rho-sol} treats the special case $d\Lambda=\lambda\,dt$, in which the existence of $w$ comes from Theorem~\ref{thm:V-S-sol}.
In \cite[Theorems 3.1 and 3.2]{wang2025fluid}, the same equation for $w$ characterizes the fluid limit of the $M_t/G/n/n$ queue, with $K=1-G$ and $F$ the fluid occupancy of the customers present at time $0$, and existence of $w$ is obtained from the queueing limit.

\begin{theorem}\label{prop:thinning}
Let $K: [0,\infty)\rightarrow \R_+$ be absolutely continuous with $K(0)=1$ and $k:=K'\leq0$. Let
\beq\label{eq:rho-form-thinning}
	\rho=F+K*dA,\quad \text{where} \quad (K*dA)(t)=\int_0^tK(t-u)\,dA(u).
\eeq
Here $A\in\cd([0,\infty),\R)$ is nondecreasing with $A(0)=0$, $F\in\cd([0,\infty),\R)$ is nonincreasing with $F(0)\leq1$, and neither $A$ nor $F$ has a singular-continuous part. Then the regulator is a thinning of the arrivals, that is, there is a measurable $w:[0,\infty)\to[0,1]$ with
\begin{equation}\label{eq:thinning}
	d\eta=(1-w)\,dA,\qquad \hat\rho=F+\int_0^\cdot K(\cdot-u)\,w(u)\,dA(u),
\end{equation}
where $w=1$ $dA$-a.e.\, on $\{\hat{\rho}<1\}$ and, on $\{\hat\rho=1\}$, the balance $\int_0^tK(t-u)\,w(u)\,dA(u)=1-F(t)$ holds. On the absolutely continuous part, we have $w\lambda=-F'-k*(w dA)$ a.e.\, on $\{\hat{\rho}=1\}$, where $\lambda$ is the density of the absolutely continuous part of $dA$. At an atom $t$ of $A$,
\begin{equation}\label{eq:atom:eta+w}
\Delta\eta(t)=\bigl(\hat\rho(t^-)+\Delta A(t)+\Delta F(t)-1\bigr)^+,\qquad w(t)=1-\frac{\Delta\eta(t)}{\Delta A(t)}\in[0,1].
\end{equation}
\end{theorem}
\begin{proof}
By Theorem~\ref{thm:V-S-sol} the solution $(\hat\rho,\eta)$ exists and is unique. By Condition~4 of Definition~\ref{def:VSP}, the measure $d\eta$ is specified by its form on the contact set $\{\hat\rho=1\}$. It suffices to show $d\eta\ll dA$, together with the formulas for atoms and for the absolutely continuous part. Then $w = 1-d\eta/dA$ takes values in $[0,1]$,
\eqref{eq:thinning} follows from \eqref{eq:VS-reg} and \eqref{eq:rho-form-thinning}, and $w=1$ $dA$-a.e.\ on $\{\hat\rho<1\}$.

\noindent
\emph{Atoms.} As $K$ is continuous, $\Delta(K*d\eta)(t)=K(0)\Delta\eta(t)=\Delta\eta(t)$, since we assume $K(0)=1$. Therefore from \eqref{eq:VS-reg} we have
\begin{equation}\label{eq:Delta-hat-rho}
	\hat\rho(t)-\hat\rho(t^-)=\Delta\rho(t)-\Delta\eta(t)=\Delta F(t)+\Delta A(t)-\Delta\eta(t),
\end{equation}
where we have used the form of $\rho$ given by \eqref{eq:rho-form-thinning}.  

If $\Delta A(t)=0$, then since $\Delta F(t)\leq0$ and $\hat\rho(t^-)\leq1$, any $\Delta\eta(t)>0$ gives $\hat\rho(t)<1$, contradicting the complementarity property in Definition~\ref{def:VSP}. Hence $\Delta A(t)=0$ implies $\Delta \eta(t)=0$. 

If $\Delta \eta(t)>0$, by the complementarity property in Definition~\ref{def:VSP}, we must have $\hat\rho(t)=1$. Since $\hat\rho(t^-)+\Delta F(t)\leq1$, in order to ensure $\hat\rho(t)=1$, relation \eqref{eq:Delta-hat-rho} implies that we must have $\Delta\eta(t)\leq\Delta A(t)$. 

Both cases provide the form of $\Delta \eta$ in \eqref{eq:atom:eta+w}  and $0 \leq \Delta \eta(t) \leq \Delta A(t)$ gives $w(t) \in [0,1]$ .

\noindent
\emph{Continuous part.} By \eqref{eq:ibp} with $\mu_K(ds)=k(s)\,ds$ and $K(0)=1$,
\begin{equation*}
    d(K*d\eta)=d\eta+(k*d\eta)\,dt,\qquad d\rho=dF+dA+(k*dA)\,dt.
\end{equation*}
Let $\eta'$ be the density of the absolutely continuous part of $d\eta$. By \eqref{eq:VS-reg}, the absolutely continuous part of $d\hat\rho$ has density $F'+\lambda-\eta'+k*d(A-\eta)$, which equals the derivative $\hat\rho'$ a.e.\ \cite[Theorem 3.22]{folland1999real}. All but countably many points of $\{\hat\rho=1\}$ are limit points of this set, and $\hat\rho'=0$ at each such point where $\hat\rho$ is differentiable. Hence
\begin{equation}\label{eq:thin-ac}
    \eta'-\lambda=F'+k*d(A-\eta)\qquad\text{a.e. on }\{\hat\rho=1\}.
\end{equation}
Since $k\leq0$, the memory term is nonpositive if $d(A-\eta)\geq0$. 

We now prove $\eta'\leq\lambda$ a.e. by a Gronwall argument. Let $\ga:=(\eta'-\lambda)^+$. By assumption, neither $d\eta$ nor $dA$ has a singular-continuous part. Moreover the atoms of $d(A-\eta)$ are nonnegative by the previous step. Hence 
$$
d(A-\eta)\geq (\lambda - \eta')dt \geq -\ga\,dt.
$$ 
Since $F'\leq0$ and $k\leq0$, \eqref{eq:thin-ac} gives 
$$
\eta' - \lambda \leq k * d(A - \eta) \leq |k|*\ga \quad \text{a.e. on } \{\hat\rho=1\}.
$$ 
Hence $\ga\leq|k|*\ga$ a.e. on $\{\hat\rho=1\}$.
Outside of this set, $\eta'=0$ a.e. by Condition~4 of Definition~\ref{def:VSP}, and hence $\ga=0$. 
Thus $\ga\leq|k|*\ga$ a.e.
Set $Y(t):=\int_0^t \ga(s)\,ds\leq\eta(t)$. Integrating $\ga\leq|k|*\ga$ and applying Fubini's theorem gives $Y \leq|\mu_K|*Y$, where $|\mu_K|(ds)=|k(s)|\,ds$. As $Y$ is continuous and $|\mu_K|(\{0\})=0$, Lemma~\ref{lem:w-norm-bd} gives 
$$
\|Y\|_{\beta,T}\leq R_{\mu_K}(\beta)\|Y\|_{\beta,T}, \qquad T>0,
$$ 
with $R_{\mu_K}(\beta)<1$ for large $\beta$. Hence $Y=0$ on $[0,T]$ for every $T>0$, that is, $\eta'\leq\lambda$ a.e. With the atoms, this gives $d\eta\leq dA$ and $\eta'=(1-w)\lambda$ a.e. Since $d(A-\eta)=w\,dA$, \eqref{eq:thin-ac} becomes $w\lambda=-F'-k*(w\,dA)$ a.e.\ on $\{\hat\rho=1\}$.
\end{proof}

\begin{corollary}\label{thm:rho-sol}
Suppose $\rho(t)=F(t)+\int_0^t K(t-u)\lambda(u)du$, where $F$ is  absolutely continuous, nonincreasing, with $F(0) \leq 1$,  and $\lambda$ is a nonnegative locally integrable function. Let $K$ be as in Theorem~\ref{prop:thinning}. Then there is a measurable $w:[0,\infty)\to[0,1]$ such that the solution to the Volterra-Skorokhod problem with kernel $K$ for $\rho$ is
\begin{align*}
	&\hat{\rho}(t)=F(t)+\int_0^t w(u)K(t-u)\lambda(u)du\\
	&d\eta(u)=(1-w(u))\lambda(u)du,
\end{align*}
where $w$ satisfies for a.e. $t \geq 0$
\begin{equation*}
	w(t)\lambda(t)=\left\{\begin{array}{ll}
		\lambda(t), & \hat{\rho}(t)<1,\\
		-F^{\prime}(t)-\int_0^t K^{\prime}(t-u)w(u)\lambda(u)du, \, &\hat{\rho}(t)=1,
	\end{array}
	\right. \label{eq:wt}
\end{equation*}
and
\begin{equation*}
	\1_{\left\{\hat{\rho}(t)<1\right\}} \leq w(t) \leq 1 \quad \text { a.e. in }[0, \infty) .
\end{equation*}
\end{corollary}
\begin{proof}
Take $dA=\lambda\,dt$ in Theorem~\ref{prop:thinning}, and set $w:=1$ on $\{\lambda=0\}$.
\end{proof}

\begin{remark}
The thinning representation is special to the input entering $\rho$ through the reflection kernel. Suppose instead $\rho=F+\tilde K*\lambda$ is offered through a kernel $\tilde K\neq K$, while the reflection still acts through $K$. A thinning would require both $\hat\rho=\rho-K*\bigl((1-w)\lambda\bigr)$ and $\hat\rho=F+\tilde K*(w\lambda)=\rho-\tilde K*\bigl((1-w)\lambda\bigr)$. This implies
\begin{equation*}
	(K-\tilde K)*\bigl((1-w)\lambda\bigr)=0.
\end{equation*}
Let $\al:=\inf\{t:w(t)<1\}$, the left endpoint of the support of the measure $(1-w)\la\, dt$, be the first blocking time. If blocking occurs, $(1-w)\lambda\not\equiv0$, and for every $T > \al$ Titchmarsh's convolution theorem \cite{titchmarsh1926} forces $K=\tilde K$ a.e.\, on $[0,T-\al]$. Since $T$ is arbitrary, $\tilde K=K$ a.e. on $[0,\infty)$. Thus no thinning of the input reproduces the loss unless $\tilde{K}=K$. A different kernel leaves only the free-boundary description: $d\eta=0$ on $\{\hat\rho<1\}$ and $\int_0^tK(t-u)\,d\eta(u)=\rho(t)-1$ on $\{\hat\rho=1\}$, with no closed form for the contact set. The queueing pre-limit of Section~\ref{sec:queue-model} is of this type.
\end{remark}

\section{Numerical experiments}\label{sec:numeric}
\subsection{Numerical solution}
Solving the Volterra-Skorokhod problem numerically relies on the fixed point structure in Theorem~\ref{thm:V-S-sol}. Since $\Gamma_K$ is nonanticipating (Proposition~\ref{prop:reflected-process}), the fixed point can be computed forward in time. For given $\rho$ and absolutely continuous $K$, $r(ds)=q(s)\,ds$ by Remark~\ref{rem:ac-resolvent}, and $q$ can be computed from $q=-k-q*k$.
Denote $\hat{\rho}=\Gamma_K(\rho)$ and $\al=r*\rho+\rho$. Recall that
\begin{equation}\label{eq:num-hat-rho}
    \hat{\rho}=\Gamma(\al-r*\hat{\rho}).
\end{equation}
We use the explicit formula of the ordinary Skorokhod map on $(-\infty,1]$:
\begin{equation}\label{eq:num-Gamma}
    \Gamma(y)(t)=y(t)-\sup_{0\le s\le t}(y(s)-1)^+=y(t)-\eta(t).
\end{equation}
On a uniform grid $t_k=kh$, we write $\hat{\rho}_k:= \hat{\rho}(t_k)$ and $\al_k:= \al(t_k)$. Using the discretization
$$
(r*\hat{\rho})(t_k)\approx h\sum_{j=1}^k q(t_k-t_j)\hat{\rho}_j=h\sum_{j=1}^{k-1} q(t_k-t_j)\hat{\rho}_j+h q(0)\hat{\rho}_k ,
$$
we see that at step $k$, only the current value $\hat{\rho}_k$ is unknown. Let $y_k=\al_k-(r*\hat{\rho})(t_k)$ and denote $\eta_k$ the Skorokhod regulator. $\hat{\rho}_k$ can be determined from the ordinary Skorokhod map by \eqref{eq:num-Gamma}:
\begin{equation}\label{eq:num-yk-etak}
    \hat{\rho}_k=y_k-\eta_k,
\end{equation}
where $\eta_k=\max\{\eta_{k-1},\,y_k-1,\,0\}$. Set $b_k:=\al_k-h\sum_{j=1}^{k-1} q(t_k-t_j)\hat{\rho}_j,\, c:=h q(0)$. We have
\begin{equation}\label{eq:num-yk}
    y_k=b_k-c\hat{\rho}_k . 
\end{equation}
Hence each time step reduces to a scalar update: 
\begin{enumerate}
\item if $b_k-c>1+\eta_{k-1}$, then
$$y_{k}-1>\eta_{k-1}.$$
Then the regulator increases, and $\hat{\rho}_k=1, \eta_k=y_k-1$; 
\item otherwise
$\eta_k=\eta_{k-1}$ and we solve from \eqref{eq:num-yk-etak} and \eqref{eq:num-yk} that
$$
    \hat{\rho}_k=\frac{b_k-\eta_{k-1}}{1+c}.
$$
\end{enumerate}
Hence the fixed point \eqref{eq:num-hat-rho} can be constructed sequentially over the grid, without a global Picard iteration. 

\subsection{A queueing example} In this section, we illustrate our results through elementary numerical experiments. Consider a queueing system where the arrival process $A^n$ is a nonhomogeneous Poisson process with rate $n\lambda(t)$, where $\lambda(t)=\frac{2}{3}\bigl(1+\sin(\tfrac{2\pi t}{10})\bigr)$ and let the service times follow the distribution $G$, where $G$ is Lognormal with parameters $(\mu, \si)=(-0.5,2)$. By Theorem~\ref{thm:fluid}, $\Gamma_K(\rho)$ with $K=1-G$ is the fluid limit of this $M_t/G/n/n$ loss queue, where $\rho$ is the fluid limit \eqref{eq:ggn} of the $M_t/G/\infty$ queue. Figure~\ref{fig:fluid} presents $\rho$ and $\Gamma_K(\rho)$, computed from the fixed point problem of Theorem~\ref{thm:V-S-sol} and from the Volterra integral equation (VIE) of Corollary~\ref{thm:rho-sol}. The two computations agree.
\begin{figure}[h]
	\centering
	\includegraphics[width=0.6\linewidth]{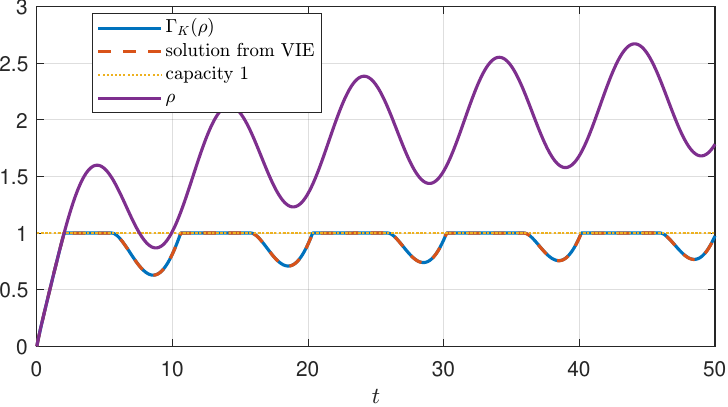}
	\caption{Fluid limit computed from the fixed point problem and from the VIE}
	\label{fig:fluid}
\end{figure}

We next investigate the evolution of the pre-limit systems. Figure~\ref{fig:Compare} shows simulated $M_t/G/\infty$ and the $M_t/G/n/n$ systems for $n=10$, $50$ and $250$ with the same arrival process and service distribution above. In Figures~\ref{fig:Compare_10}, \ref{fig:Compare_50} and \ref{fig:Compare_250} we compare the corresponding scaled number in the system $\rho^n$ and $\bar{N}^n$ with the reflection $\Gamma_K(\rho^n)$. The simulated paths $\rho^n$ and $\bar{N}^n$ are piecewise constant. The reflection $\Gamma_K(\rho^n)$ is piecewise smooth, because the departures of blocked customers are replaced by their mean flow as  discussed in Remark~\ref{rk:approx}. In Figures~\ref{fig:Dif_10}, \ref{fig:Dif_50} and \ref{fig:Dif_250} we present the difference between $\bar{N}^n$ and $\Gamma_K(\rho^n)$. As $n$ grows, this difference shrinks and both processes approach the fluid limit in Figure~\ref{fig:fluid}. This supports the observation in Theorem~\ref{thm:fluid} and Remark~\ref{rk:approx}.
\renewcommand{\thesubfigure}{\alph{subfigure}}
\begin{figure}[h]
	\centering
	\begin{subfigure}{0.45\textwidth}
		\centering
		\includegraphics[width=\linewidth]{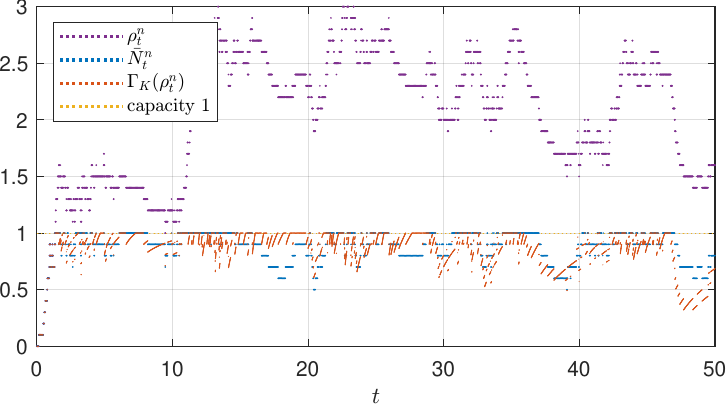}
		\caption{(a) $n=10$: $\rho^n$, $\bar N^n$, $\Gamma_K(\rho^n)$}
		\label{fig:Compare_10}
	\end{subfigure}
	\begin{subfigure}{0.45\textwidth}
		\centering
		\includegraphics[width=\linewidth]{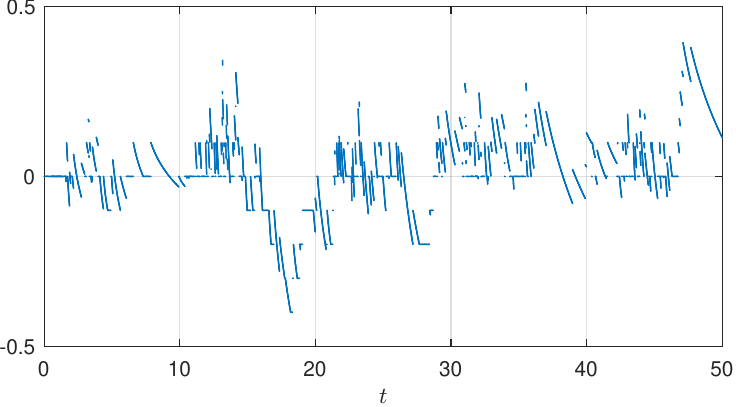}
		\caption{(b) $n=10$: $\bar N^n-\Gamma_K(\rho^n)$}
		\label{fig:Dif_10}
	\end{subfigure}
	\begin{subfigure}{0.45\textwidth}
		\centering
		\includegraphics[width=\linewidth]{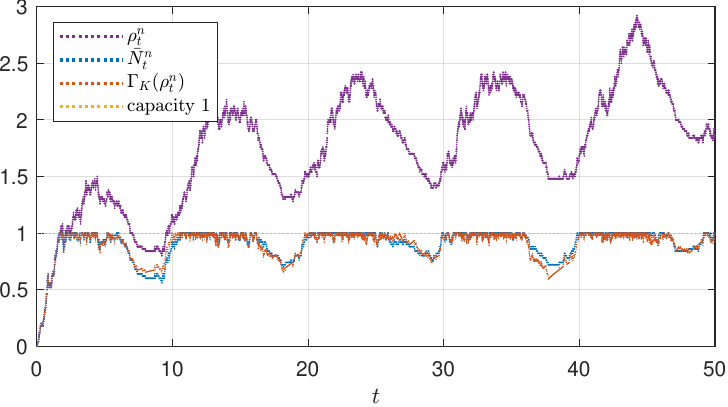}
		\caption{(c) $n=50$: $\rho^n$, $\bar N^n$, $\Gamma_K(\rho^n)$}
		\label{fig:Compare_50}
	\end{subfigure}
	\begin{subfigure}{0.45\textwidth}
		\centering
		\includegraphics[width=\linewidth]{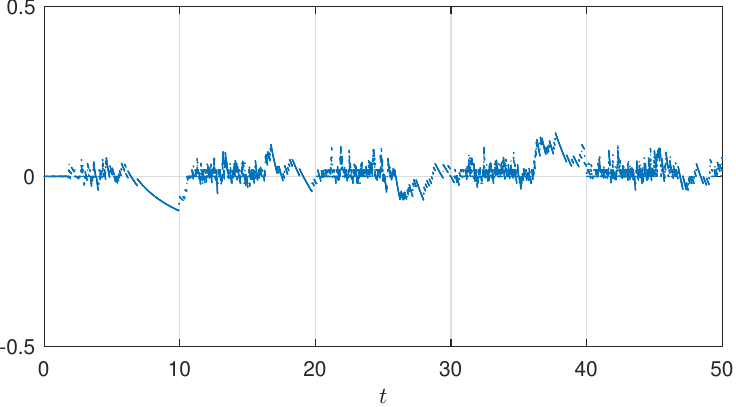}
		\caption{(d) $n=50$: $\bar N^n-\Gamma_K(\rho^n)$}
		\label{fig:Dif_50}
	\end{subfigure}
		\begin{subfigure}{0.45\textwidth}
		\centering
		\includegraphics[width=\linewidth]{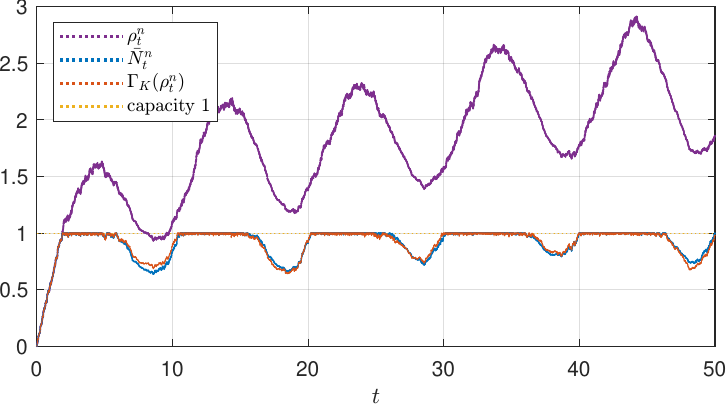}
		\caption{(e) $n=250$: $\rho^n$, $\bar N^n$, $\Gamma_K(\rho^n)$}
		\label{fig:Compare_250}
	\end{subfigure}
	\begin{subfigure}{0.45\textwidth}
		\centering
		\includegraphics[width=\linewidth]{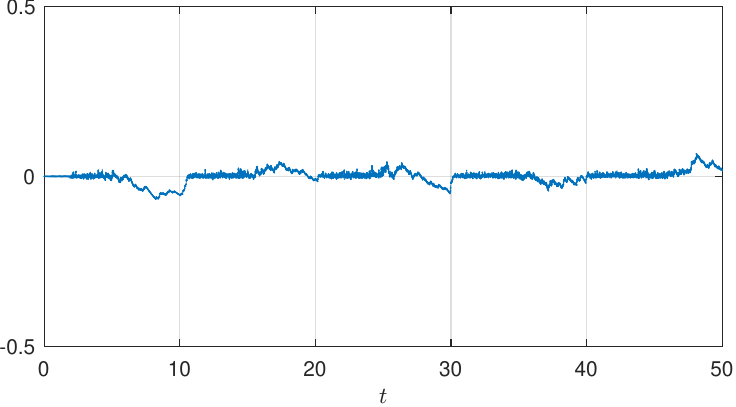}
		\caption{(f) $n=250$: $\bar N^n-\Gamma_K(\rho^n)$}
		\label{fig:Dif_250}
	\end{subfigure}
	\caption{Comparison between $M_t/G/\infty$, $M_t/G/n/n$ and reflection of $M_t/G/\infty$}
	\label{fig:Compare}
\end{figure}

In Figure~\ref{fig:Compare_Gt} we present similar experiments for the $G_t/G/\infty$ and $G_t/G/n/n$ systems, with arrivals generated by a time-changed renewal process. Specifically, let $R$ be a renewal counting process whose interarrival times are independent and identically distributed according to an Erlang-2 distribution with rate $2$, having mean $1$ and variance $1/2$. We define the arrival process by
$$
A^n(t)=R\left(n\Lambda(t)\right),\quad\Lambda(t)=\int_0^t\lambda(s)ds,
$$
with $\lambda$ as above. This yields non-Poisson arrivals with $\bar{A}^n \to \Lambda$ as in Assumption~\ref{asm:queue0}, where $\Lambda$ is the same cumulative rate of the nonhomogeneous Poisson example. We retain the same service distribution, and use identical arrival times and service requirements for both the infinite-server and loss systems. Figures~\ref{fig:Compare_Gt_Gamma2_n10}, \ref{fig:Compare_Gt_Gamma2_n50} and \ref{fig:Compare_Gt_Gamma2_n250} compare the scaled occupancies $\rho^n$ and $\bar{N}^n$ with the reflection $\Gamma_K(\rho^n)$, and Figures~\ref{fig:Dif_Gt_Gamma2_n10}, \ref{fig:Dif_Gt_Gamma2_n50} and \ref{fig:Dif_Gt_Gamma2_n250} show the difference $\bar{N}^n-\Gamma_K(\rho^n)$. As in Figure~\ref{fig:Compare}, the difference shrinks as $n$ grows.
\begin{figure}[h]
	\centering
	\begin{subfigure}{0.45\textwidth}
		\centering
		\includegraphics[width=\linewidth]{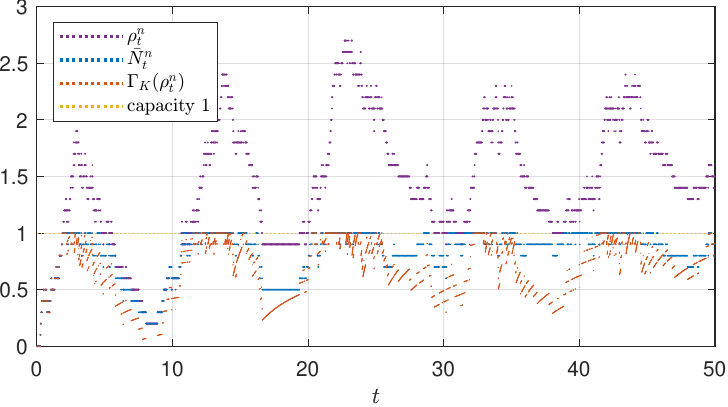}
		\caption{(a) $n=10$: $\rho^n$, $\bar N^n$, $\Gamma_K(\rho^n)$}
		\label{fig:Compare_Gt_Gamma2_n10}
	\end{subfigure}
	\begin{subfigure}{0.45\textwidth}
		\centering
		\includegraphics[width=\linewidth]{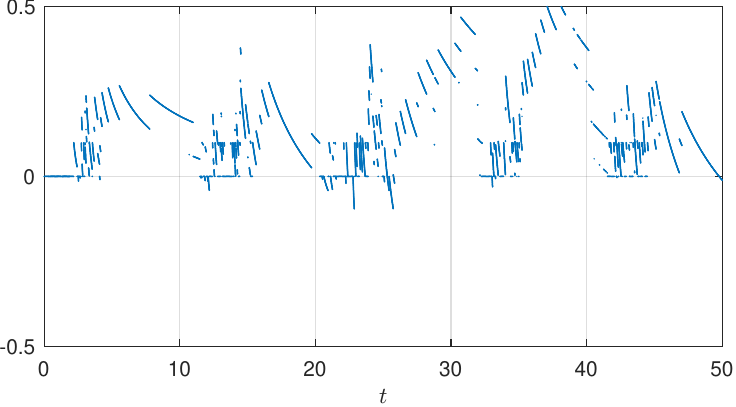}
		\caption{(b) $n=10$: $\bar N^n-\Gamma_K(\rho^n)$}
		\label{fig:Dif_Gt_Gamma2_n10}
	\end{subfigure}
	\begin{subfigure}{0.45\textwidth}
		\centering
		\includegraphics[width=\linewidth]{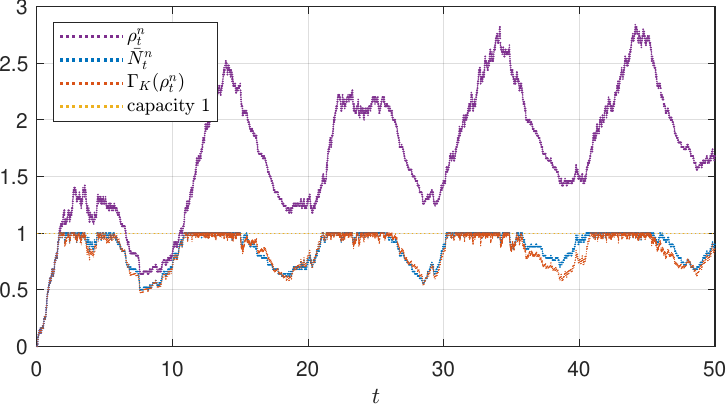}
		\caption{(c) $n=50$: $\rho^n$, $\bar N^n$, $\Gamma_K(\rho^n)$}
		\label{fig:Compare_Gt_Gamma2_n50}
	\end{subfigure}
	\begin{subfigure}{0.45\textwidth}
		\centering
		\includegraphics[width=\linewidth]{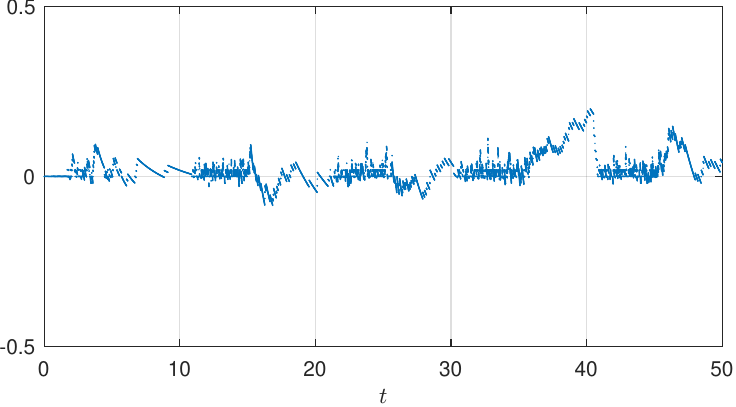}
		\caption{(d) $n=50$: $\bar N^n-\Gamma_K(\rho^n)$}
		\label{fig:Dif_Gt_Gamma2_n50}
	\end{subfigure}
		\begin{subfigure}{0.45\textwidth}
		\centering
		\includegraphics[width=\linewidth]{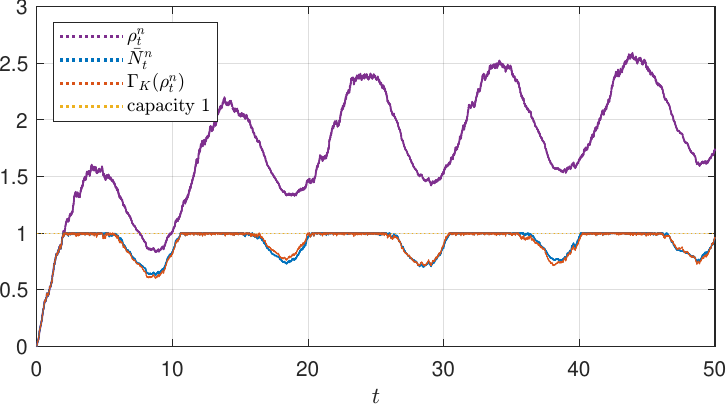}
		\caption{(e) $n=250$: $\rho^n$, $\bar N^n$, $\Gamma_K(\rho^n)$}
		\label{fig:Compare_Gt_Gamma2_n250}
	\end{subfigure}
	\begin{subfigure}{0.45\textwidth}
		\centering
		\includegraphics[width=\linewidth]{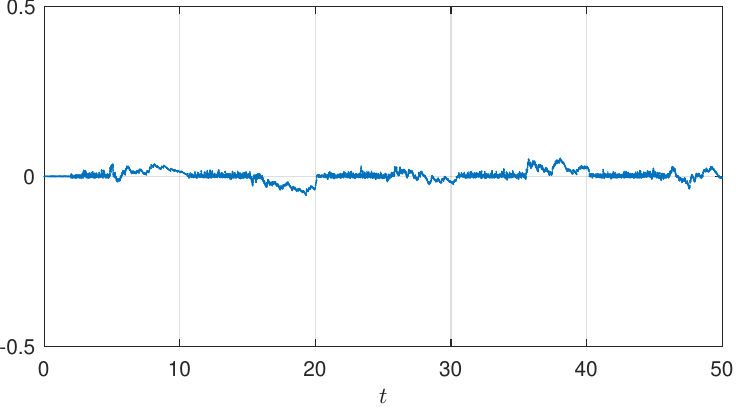}
		\caption{(f) $n=250$: $\bar N^n-\Gamma_K(\rho^n)$}
		\label{fig:Dif_Gt_Gamma2_n250}
	\end{subfigure}
	\caption{Comparison between $G_t/G/\infty$, $G_t/G/n/n$ and reflection of $G_t/G/\infty$}
	\label{fig:Compare_Gt}
\end{figure}

\section{Conclusion}\label{sec:conclusion}

We showed that the fluid limit of the $G_t/G/n/n$ loss queue is the Volterra-Skorokhod reflection $\Gamma_K(\rho)$ of the fluid limit $\rho$ of the infinite-server queue, with $K=1-G$ the service survival function, for any arrival process satisfying a functional strong law of large numbers (Theorem~\ref{thm:fluid}). The limiting regulator is the cumulative fluid loss. It is a thinning of the arrivals (Theorem~\ref{prop:thinning}, Corollary~\ref{thm:rho-sol}) and increases with the arrival rate (Corollary~\ref{cor:loss-order}). The experiments of Section~\ref{sec:numeric} illustrate both the fixed-point/VIE agreement and the pre-limit approximation $\Gamma_K(\rho^n)\approx\bar N^n$.

These results rest on the Volterra-Skorokhod reflection problem, in which the regulator is convolved with the memory kernel $K$ rather than added to the state, and on its well-posedness. For every c\`adl\`ag input and every nonnegative c\`adl\`ag kernel $K$ of bounded variation with $K(0)=1$ the problem has a unique solution (Theorem~\ref{thm:V-S-sol}), obtained by reducing the convolved problem to an ordinary Skorokhod map through the resolvent of $-\mu_K$ and a Bielecki fixed point. The map $\Gamma_K$ is Lipschitz in the uniform norm (Proposition~\ref{prop:lip}) and nonanticipating (Proposition~\ref{prop:reflected-process}). It therefore acts pathwise on stochastic inputs and turns $\hat\rho=\Gamma_K(\rho)$ into an adapted reflected stochastic process.

The convolution changes the qualitative behavior of the reflection. The occupancy $\hat\rho$ is not monotone in the input (Example~\ref{prop:non-mono}), unlike the classical map with $K\equiv1$. Yet, for nonincreasing $K$, the cumulative regulator stays monotone (Corollary~\ref{thm:eta-mono}) because it is the minimal admissible regulator (Theorem~\ref{prop:eta-min}). This passes to a stochastic order and an almost-sure coupling of the loss (Corollary~\ref{cor:eta-st}). If the input enters through a kernel other than the reflection kernel, the thinning representation fails, leaving only the free-boundary description on the contact set $\{\hat\rho=1\}$.

\bigskip

\bibliographystyle{plain}

\bibliography{ref}

\end{document}